\documentclass[a4paper]{amsart}
\usepackage{latexsym,bm,stmaryrd}
\usepackage{amsmath,amsthm,amsfonts,amssymb,mathrsfs,pb-diagram}
\usepackage[a4paper,hmargin=26mm,vmargin=28mm]{geometry}

\usepackage{microtype}
\usepackage{mathtools}
\usepackage{mathbbol,wasysym}

\usepackage{tikz}
\usetikzlibrary{arrows,decorations.markings}

\advance\textheight by 1.5mm
\usepackage{graphics}
\usepackage[all]{xy}

\usepackage[enableskew,vcentermath]{youngtab}
\def\tab(#1){\mbox{\small$\young(#1)$}\,}

\title[Symmetric structure on the endomorphism algebra of projective-injective module]
{Symmetric structure for the endomorphism algebra of projective-injective module in parabolic category}
\author{Jun Hu}\address{School of Mathematics and Statistics,
Beijing Institute of Technology,
Beijing, 100081, P.R.~China}
\email{junhu404@bit.edu.cn}

\author{Ngau Lam}
\address{Department of Mathematics, National Cheng-Kung University, Tainan, 70101, Taiwan}
\email{nlam@mail.ncku.edu.tw}

\usepackage[capitalise,nameinlink]{cleveref}
\crefname{equation}{}{}

\DeclareFontFamily{OT1}{pzc}{}
\DeclareFontShape{OT1}{pzc}{m}{it}{<-> s * [1.20] pzcmi7t}{}
\DeclareMathAlphabet{\mathpzc}{OT1}{pzc}{m}{it}

\DeclareSymbolFontAlphabet{\mathbb}{AMSb}

\usepackage{tikz}
\usetikzlibrary{matrix}

\let\<=\langle
\let\>=\rangle
\def\({\big(}
\def\){\big)}
\def\Z{\mathbb{Z}}
\def\C{\mathbb{C}}

\def\g{\mathfrak{g}}
\def\p{\mathfrak{p}}
\def\b{\mathfrak{b}}
\def\h{\mathfrak{h}}
\def\O{\mathcal{O}}

\def\lam{\lambda}
\def\Lam{\Lambda}
\def\hom{\mathrm{hom}}

\DeclareMathOperator\gmod{\!{-}gmod}
\DeclareMathOperator\Hom{Hom}
\DeclareMathOperator\End{End}
\DeclareMathOperator\Ext{Ext}
\DeclareMathOperator\soc{soc}
\DeclareMathOperator\im{Im}
\DeclareMathOperator\head{head}
\DeclareMathOperator\rad{rad}

\DeclareMathOperator\tr{tr}
\DeclareMathOperator\id{id}
\DeclareMathOperator\Id{Id}
\DeclareMathOperator\op{op}

\DeclareMathOperator\Ker{Ker}

\DeclareMathOperator\pr{pr}

\newcommand\For{\underline{\mathrm{For}}}

\newcommand\wPFlat[1][\mu]{\widetilde{P}_\flat^{#1}}
\newcommand\wLFlat[1][\mu]{\widetilde{L}_\flat^{#1}}

\newcommand\PFlat[1][\mu]{P_\flat^{#1}}
\newcommand\LFlat[1][\mu]{L_\flat^{#1}}
\newcommand\DFlat[1][\mu]{\Delta_\flat^{#1}}

\def\mod{\text{{-}mod}}
\newcommand\Pw{\widetilde{P}^{w\cdot\lam}}

\newcommand\WLp{{\mathring{W}}^\psi}
\newcommand\WLl{{\mathring{W}}^\lam}

\newcommand\fw{{\overline{\theta}_w}}
\newcommand\fwn{{\overline{\theta}}}

 \newif\ifdraft\drafttrue
 \usepackage{todonotes}

\swapnumbers
\newcommand{\secref}[1]{Section~\ref{#1}}
\numberwithin{equation}{section}
\newtheorem{prop}[equation]{Proposition}
\newtheorem{thm}[equation]{Theorem}
\newtheorem{cor}[equation]{Corollary}
\newtheorem{lem}[equation]{Lemma}
\newtheorem{problem}[equation]{Conjecture}
\theoremstyle{definition}
\newtheorem{dfn}[equation]{Definition}
\theoremstyle{remark}
\newtheorem{rem}[equation]{Remark}

\begin{document}

\begin{abstract}
  We show that for any singular dominant integral weight $\lam$ of a complex semisimple Lie algebra $\mathfrak{g}$, the endomorphism algebra $B$ of any projective-injective module of the parabolic BGG category $\mathcal{O}_\lam^{\p}$ is a symmetric algebra (as conjectured by Khovanov) extending the results of Mazorchuk and Stroppel for the regular dominant integral weight. Moreover, the endomorphism algebra $B$ is equipped with a homogeneous (non-degenerate) symmetrizing form. In the appendix, there is a short proof due to K. Coulembier and V. Mazorchuk showing that the endomorphism algebra $B_\lam^{\p}$ of the basic projective-injective module of $\mathcal{O}_\lam^{\p}$  is a symmetric algebra.
\end{abstract}

\maketitle

\section{Introduction}
Symmetric algebra is an important class of algebras enjoying many good properties. The purpose of this paper is to study the symmetric structure on  the endomorphism algebra of any projective-injective (i.e. at the same time projective and injective) module in an integral block of a parabolic Bernstein-Gelfand-Gelfand (BGG) category. In the first part of this paper we study whether the endomorphism algebra $B$ of the basic projective-injective module over any finite-dimensional algebra $A$ is a symmetric algebra. The algebra $B$ can be equipped with a non-degenerate associative bilinear form $(-,-)_{\tr}$ so that $B$ is a Frobenius algebra, see Proposition \ref{symmetric} below for the precise statement. Furthermore, every indecomposable projective-injective module having isomorphic head and socle is a necessary condition for the algebra $B$ to be a symmetric algebra, see Lemma~\ref{dual1} and \ref{keylem1} below for the precise statements.

Now we assume that the head and the socle of every indecomposable projective-injective module are isomorphic and the algebra $A$ is positively graded. The definition of $(-,-)_{\tr}$ mentioned above relies on a prefixed basis of the algebra. A priori, it is unclear whether the non-degenerate associative bilinear form $(-,-)_{\tr}$ is symmetric or not. In order to characterize the algebra $B$ being a symmetric algebra, we propose a notion, called ``admissible condition", on the homogeneous bases of $B$. Roughly speaking, the admissible condition are certain symmetric conditions for the multiplication of the homogeneous basis of the algebra $B$. Under certain circumstances in the $\Z$-graded setting, we show that $B$ is a symmetric algebra if and only if there exists an admissible basis of $B$, see Corollary \ref{symm-eq} below. In this case, the form $(-,-)_{\tr}$ is symmetric, see Proposition \ref{keylem3} below. Therefore every indecomposable projective-injective module in each block of $B$ having the same graded length  is a necessary condition for the algebra being a symmetric algebra. In Proposition~\ref{keylem4} we show that there are some interesting classes of algebras such that every indecomposable projective-injective module in each block of the algebra has the same graded length.

Let $\g$ be a complex semisimple Lie algebra with a fixed Borel subalgebra $\b$ containing the Cartan subalgebra $\h$ and $\p$ a parabolic subalgebra containing the fixed Borel subalgebra $\b$. For a dominant integral weight $\lam$, let $\O^\p_\lam$ denote the subcategory of the parabolic Bernstein-Gelfand-Gelfand (BGG) category $\O^\p$\cite{Ro} with respect to the parabolic subalgebra $\p$ consisting of modules whose irreducible subquotients have highest weights belonging to the orbit of $\lam$ under the dot action of the Weyl group. The self-dual projective modules in $\O^\p$ have been studied intensively in \cite{Irv1, Irv2, IrvShelton}. Note that the notion of the self-dual projective module is the same as projective-injective   module (see Lemma~\ref{socular2} below). In those works, Irving and his collaborators attempted to speculate that every self-dual indecomposable projective module in $\O_\lam^\p$ always has the same Loewy length and hence has the same graded length. This speculation was proved by Mazorchuk and Stroppel for any regular dominant integral weight $\lam$ \cite[Theorem 5.2, Remark 5.3]{MazorStrop}  and by Coulembier and Mazorchuk  in all cases \cite{CoulMazor}.

\begin{problem} \label{op} \text{(Khovanov)} For a dominant integral weight $\lam$, the endomorphism algebra of each projective-injective module in $\O_\lam^p$ is a symmetric algebra.
\end{problem}

The Conjecture \ref{op} was proposed by Khovanov formulated in the beginning of Section~5 of \cite{MazorStrop}.  Mazorchuk and Stroppel  proved the conjecture for the basic projective-injective module in $\O_\lam^\p$ and regular dominant integral weight $\lam$ \cite[Theorem 4.6]{MazorStrop} and hence the endomorphism algebra of any projective-injective module in $\O_\lam^p$ is a symmetric algebra. The purpose of this paper is to give a proof of the above conjecture for the remaining case. That is, the conjecture holds for any singular dominant integral weight $\lam$. The main result of this paper is the following theorem.

\begin{thm} \label{mainthm1} For any singular dominant integral weight $\lam$, the endomorphism algebra of any projective-injective module in $\O_\lam^p$ is a symmetric algebra.
\end{thm}

We sketch an outline of our proof. The result of Mazorchuk and Stroppel \cite[Theorem 4.6]{MazorStrop}
ensures that there exists a symmetrizing form on the endomorphism algebra of the basic projective-injective module in $\O_0^\p$. By the general theory developed in Section 3, there exists an admissible basis for the endomorphism algebra of the basic projective-injective module in $\O_0^\p$ and the attached canonical form $\tr$ is a symmetrizing form. For a singular dominant integral weight $\lam$, we apply the graded translation functor to connect $\O_0^\p$ with $\O_\lam^\p$. Then we use the crucial endomorphism $\fwn$ to define a form $\tr_\lam$ on the endomorphism algebra of the basic projective-injective module of $\O_\lam^\p$ through the graded translation functor  and the canonical form $\tr$. As a result, we can show with help of the technical Proposition \ref{keyprop} that $\tr_\lam$ is a homogeneous symmetrizing form.

After a first version of the paper was submitted, Professor Volodymyr Mazorchuk sent us a short proof of Theorem \ref{mainthm1} due to Coulembier and himself. We appreciate his patience to explain their very clever argument in the short proof to us and their kindness to allow us to include the proof as an appendix of this paper.

\smallskip

The paper is organized as follows. In Section 2 we work in a general setting to study the endomorphism algebra $B$ of the basic projective-injective module over any finite-dimensional algebra $A$. We first show in Proposition \ref{symmetric} that if the algebra $A$ is equipped with an anti-involution fixing each simple $A$-module then $B$ is endowed with a non-degenerate form ``$\tr$" (equivalently, a non-degenerate associative bilinear form $(-,-)_{\tr}$). In particular,  $B$ is a Frobenius algebra. The canonical form``$\tr$" on $B$ depends on the choice of a prefixed appropriate  basis (Definition \ref{appropriate basis}) of $B$. In Section 3 we work in a $\Z$-graded setting and propose a notion of admissible basis in Definition \ref{HAdmissibleCond}. We show in Proposition \ref{keylem3} that the canonical form attached to an admissible basis is always symmetric. In Corollary \ref{symm-eq} we show that under certain circumstances in a $\Z$-graded setting $B$ is symmetric if and only if there is an admissible basis for $B$. In Proposition~\ref{keylem4} we show that with certain special assumptions in the $\Z$-graded setting (which are satisfied in the parabolic BGG category case), all the indecomposable projective-injective modules in any fixed block of certain finite-dimensional algebra always have the same graded length. In Section 4 we give the proof of Theorem \ref{mainthm1}. The whole section is devoted to showing the existence of the endomorphism $\fwn$ and to developing some useful properties of $\fwn$ described in Proposition~\ref{keyprop}. The properties of $\fwn$  are  the main ingredients of the proof of Theorem~\ref{mainthm1}. Finally,  a short proof (due to Coulembier and Mazorchuk) of Theorem \ref{mainthm1} is in the appendix.

\bigskip

\noindent {\it Notations:} We let $\Z$, $\Z_+$ and $\Z_{>0}$ denote the sets of all, non-negative and positive integers, respectively. Let $\C$ denote the field of complex numbers. The identity map of a set $X$ is denoted by $\id_X$. The image of a function $f$ is denoted by $\im(f)$.

\bigskip

\section{Endomorphism algebras of projective-injective modules}\label{End}

The purpose of this section is to study the endomorphism algebra of the basic projective-injective module over any finite-dimensional algebra in a general setting.

Let $A$ be a finite-dimensional unital $K$-algebra over an algebraically closed field $K$ and the category of finite-dimensional $A$-modules is denoted by $A\mod$. Let $\{L^\lam|\lam\in\Lam\}$ denote a complete set of pairwise non-isomorphic simple $A$-modules, where $\Lam$ is a finite index set. From our assumption that  $K$ is algebraically closed, we have that
 \begin{equation}\label{scalarendo}
\text{$\End_A(L^\lam)=K$, \quad for any $\lam\in\Lam$}.
\end{equation}
For each $\lam\in\Lam$, let $P^\lam$ and $I^\lam$ denote the projective cover and the injective envelope of $L^\lam$, respectively.
For $M\in A\mod$, the socle of $M$, denoted by $\soc (M)$, is the
largest semisimple submodule of $M$ and the radical of $M$, denoted by $\rad(M)$, is the smallest
submodule such that $M /\rad (M)$ is semisimple.  We will call $M /\rad (M)$, denoted by $\head(M)$, the head of $M$.

We let
$$
\Lam_0:=\{\lam\in\Lam\,|\,\text{$P^\lam$ is an injective module in $A\mod$}\}.
$$
Therefore $\Lam_0$ parameterizes all the indecomposable projective-injective modules in $A\mod$. Henceforth, we assume that $\Lam_0\neq\emptyset$. For every $\lam\in\Lam_0$, there is a unique $\lam'\in\Lam$ such that $P^\lam\cong I^{\lam'}$ and $\soc(P^\lam)\cong L^{\lam'}$ since $P^\lam$ is injective. It is clear that the map $'$ from $\Lam_0$ to the set $\Lam$ is injective.

\begin{dfn} The module $\bigoplus_{\lam\in\Lam_0}P^\lam$ is called the basic projective-injective module of $A\mod$. We define
\begin{equation}\label{Hecke}
B:=\End_{A}\Bigl(\bigoplus_{\lam\in\Lam_0}P^\lam\Bigr)=\bigoplus_{\lam,\mu\in\Lam_0}\Hom_A(P^\lam,P^\mu).
\end{equation}
\end{dfn}

We assume that the $K$-algebra $A$ is equipped with a $K$-linear anti-involution $\ast$. For each $M\in A\mod$, we define the dual module $M^\ast$ of $M$ as follows: $M^\ast:=\Hom_{K}(M,K)$ as a $K$-vector space, and $(af)(m):=f(a^\ast m)$ for any $a\in A$, $m\in M$ and $f\in M^\ast$. Then the anti-involution $\ast$ defines a contravariant exact functor from $A\mod$ to itself. It is clear that the functor $\ast$ gives an equivalence of categories since $(M^*)^*\cong M$ for all $M\in A\mod$. In particular, $(L^\lam)^\ast$ is an irreducible module for every $\lam\in\Lam$. Furthermore, $(I^\lam)^\ast$ and $(P^\lam)^\ast$, $\lam\in\Lam$, are the projective cover and the injective envelope of $(L^\lam)^\ast$, respectively.

\begin{lem} \label{dual1} Assume that the $K$-algebra $A$ is equipped with a $K$-linear anti-involution $\ast$ satisfying $(L^\lam)^\ast\cong L^\lam$ for any $\lam\in\Lam$. Then the map $'$ defines an involution on the set $\Lam_0$. Moreover, for any $\lam\in\Lam_0$, we have $$P^{\lam}\cong I^{\lam'}\cong (P^{\lam'})^\ast\qquad \text{and} \qquad P^{\lam'}\cong I^\lam\cong (P^\lam)^\ast.$$
\end{lem}

\begin{proof}
Since $(L^\lam)^\ast\cong L^\lam$ for any $\lam\in\Lam$, we have $(I^\lam)^\ast\cong P^\lam $ and $ (P^\lam)^\ast \cong I^\lam$. Therefore $ P^{\lam'}\cong (I^{\lam'})^\ast\cong (P^\lam)^\ast$ for all $\lam\in\Lam_0$ and hence $P^{\lam'}$ is projective and injective for every $\lam\in\Lam_0$. It follows that $\lam'\in\Lam_0$ and $(\lam')'=\lam$  for any $\lam\in\Lam_0$ since $(M^*)^*\cong M$ for all $M\in A\mod$. Thus the map $'$ defines an involution on the set $\Lam_0$. This completes the proof.
\end{proof}

From now on, we assume that the map $'$ is an involution on $\Lam_0$ throughout this section. Thus we have $\soc(P^{\lam'})=\soc(I^{\lam})\cong L^\lam$ and hence $\Hom_{A}(P^\lam,\soc(P^{\lam'}))$ is one-dimensional by \eqref{scalarendo}. Therefore for each $\lam\in\Lam_0$ the subspace of all the homomorphisms $f\in\Hom_{A}(P^\lam,P^{\lam'})$ satisfying $\im(f)=\soc(P^{\lam'})$ is one-dimensional. For each $\lam\in\Lam_{0}$, there is a unique (up to a scalar) nonzero homomorphism
\begin{equation}\label{theta}
\theta_{\lam}\in\Hom_{A}(P^\lam,P^{\lam'}) \quad\text{satisfying\quad $\im(\theta_{\lam})=\soc(P^{\lam'})$.}
 \end{equation}
 We will fix a $\theta_{\lam}$ for each $\lam\in\Lam_0$.

\begin{dfn} \label{appropriate basis}  Assume that the map $'$ is an involution on $\Lam_0$. A $K$-basis $\Upsilon$ of $B$ is said to be appropriate if $$\bigsqcup_{\lam,\mu\in\Lam_0}\Hom_A(P^\lam,P^\mu)\supseteq \, \Upsilon\, \supseteq \{\theta_\lam\,|\,\lam\in\Lam_0\}.$$
\end{dfn}

A $K$-linear map $\tau$ from a $K$-algebra $R$ to $K$ is called a form on $R$. The form $\tau$ induces an associative bilinear form $(-,-)_{\tau}$ on $R$ defined by $(f,g)_{\tau}:=\tau(fg)$ for all $f,g\in R$. Recall that a bilinear form  $(-,-)$ on $R$ is called associative if  $(fh,g)=(f,hg)$ for all $f,g,h\in R$.
The form $\tau$ is called symmetric (resp., non-degenerate) if the bilinear form $(-,-)_{\tau}$ is symmetric (resp., non-degenerate). The form $\tau$ is called a symmetrizing form on $R$ if the bilinear form $(-,-)_{\tau}$ is a non-degenerate symmetric form. $R$ is called a symmetric algebra if $R$ is equipped with a symmetrizing form.

Given an appropriate $K$-basis of $B$ of the following form
\begin{equation}\label{fixedbasis}
\Upsilon:=\bigcup_{\lam,\mu\in\Lam_0}\Bigl\{f_j^{\lam,\mu}\Bigm|f_j^{\lam,\mu}\in\Hom_{A}(P^\lam,P^\mu),\,\,1\leq j\leq\dim\Hom_{A}(P^\lam,P^\mu)\Bigr\},
\end{equation}
we define a form $\tr$ on $B$ determined by
\begin{equation}\label{tr0}
\tr(f_{j}^{\lam,\mu}):=\begin{cases} 1, &\text{if $\mu=\lam'$ and $f_j^{\lam,\mu}=\theta_{\lam}$;}\\
0, &\text{if $\mu\neq\lam'$, or $\mu=\lam'$ and $f_j^{\lam,\mu}\neq\theta_\lam$.}
\end{cases}
\end{equation}

\begin{dfn} \label{canonicalform} The form $\tr$ defined in (\ref{tr0}) is called the canonical form attached to the appropriate basis $\Upsilon$ of $B$.
\end{dfn}

\begin{lem} \label{keylemma0}  Assume that the map $'$ is an involution on $\Lam_0$.  For $\lam,\mu\in\Lam_0$ and $0\neq f\in\Hom_A(P^\lam,P^{\mu})$, there exist $g\in\Hom_A(P^{\mu},P^{\lam'})$ and
$h\in\Hom_A(P^{{\mu'}},P^\lam)$ such that $gf=\theta_\lam$ and $fh=\theta_{{\mu'}}$.
\end{lem}

\begin{proof} Set $N:=f^{-1}(\soc(P^{\mu}))$. The restriction map $f|_{N}\,:\, N\longrightarrow \soc(P^{\mu})$ of $f$ on $N$ is surjective because the image of the nonzero homomorphism $f$ contains $\soc(P^{\mu})\cong L^{{\mu'}}$. Since $P^{{\mu'}}$ is the projective cover of $\soc(P^{\mu})\cong L^{{\mu'}}$ and $f|_{N}$ is an epimorphism, we can find a homomorphism $h_1: P^{{\mu'}}\rightarrow N$ such that the image of $f|_{N}h_1$ is $\soc(P^{\mu'})$. Now $h_1$ is regarded as a homomorphism from $P^{{\mu'}}$ to $P^{\lam}$ and hence $fh_1=c\theta_{{\mu'}}$ for some nonzero $c\in K$. Therefore we have $fh=\theta_{{\mu'}}$ by choosing $h=c^{-1}h_1$.

Dually, we set $M:=f(P^\lam)$. We have a natural embedding $\iota: M\hookrightarrow P^{\mu}$. Since $P^\lam$ has a unique simple head $L^\lam$, it follows that $M$ has a unique simple head $L^\lam$ too. Let $\pi: M\rightarrow P^{\lam'}$ be a homomorphism which sends $M$ onto the unique simple socle $L^\lam$ of $P^{\lam'}$. Since $P^{\lam'}$ is injective, we can find a homomorphism $g_1\in\Hom_{A}(P^{\mu},P^{\lam'})$ such that $g_1\iota =\pi $. Therefore $g_1f=\pi f=c\theta_\lam$ for some nonzero $c\in K$. Now we take $g=c^{-1}g_1$, then $gf=\theta_{\lam}$ as required. This completes the proof of the lemma.
\end{proof}

\begin{prop} \label{symmetric} If the map $'$ is an involution on $\Lam_0$,  then the bilinear form $(-,-)_{\tr}$ induced by $\tr$ is non-degenerate. In other words, $B$ is a Frobenius algebra over $K$. In particular, if the $K$-algebra $A$ is equipped with a $K$-linear anti-involution $\ast$ satisfying $(L^\lam)^\ast\cong L^\lam$ for any $\lam\in\Lam$, then $B$ is a Frobenius algebra over $K$.
\end{prop}

\begin{proof} Let $f=\sum_{\lam,\mu\in\Lam_0} f_{\lam, \mu}$ be a nonzero element in $B$ such that $f_{\lam, \mu}\in\Hom_{A}(P^\lam,P^\mu)$ for all $\lam,\mu\in\Lam_0$. Let $g:=f_{\lam, \mu}$ be a nonzero component of $f$. By Lemma \ref{keylemma0}, we can find a homomorphism $h: P^{\mu'}\rightarrow P^\lam$ such that $gh=\theta_{\mu'}$. Therefore we have
$$
(f,h)_{\tr}=\tr\big((\sum_{\gamma,\beta\in\Lam_0} f_{\gamma,\beta})h\big)=\tr\big((\sum_{\beta\in\Lam_0} f_{\lam, \beta})h\big)
=\tr(f_{\lam, \mu}h)=\tr(\theta_{\mu'})=1\not=0.
$$
The third equality follows from the definition (\ref{tr0}) of the form ${\tr}$. The second part of the proposition follows from  Lemma~\ref{dual1}.
\end{proof}

In general, it is not easy to determine whether the bilinear form $(-,-)_{\tr}$ on $B$ is symmetric or not. The following lemma gives necessary conditions for a form to be a symmetrizing form on $B$.

\begin{lem} \label{keylem1}Assume the map $'$ is an involution on $\Lam_0$. If there exists a symmetrizing form $\tau$ on $B$, then we have
\begin{enumerate}
\item[(i)] $\lam'=\lam$ and $\tau(\theta_\lam)\neq 0$ for all $\lam\in\Lam_0$;
\item[(ii)] $\tau(f)=0$ for all $f\in\Hom_A(P^\lam,P^\mu)$ with $\lam\not=\mu\in\Lam_0$.
\end{enumerate}
\end{lem}

\begin{proof} Since $\tau$ is non-degenerate, for each $\gamma\in\Lam_0$ we can find a nonzero element $f=\sum_{\lam,\mu\in\Lam_0} f_{\lam, \mu}$ in $B$ such that $(\theta_\gamma,f)_\tau\neq 0$, where $f_{\lam, \mu}\in\Hom_{A}(P^\lam,P^\mu)$ for all $\lam,\mu\in\Lam_0$. Since $\tau(\theta_\gamma f)\not=0$, we have
$$
0\neq\theta_\gamma  f=\theta_\gamma \big(\sum_{\lam,\mu\in\Lam_0} f_{\lam,\mu}\big) =\theta_\gamma \big(\sum_{\lam\in\Lam_0} f_{\lam, \gamma}\big)=\theta_\gamma  f_{\gamma, \gamma}.
$$
The last equality follows from the fact that there is no nonzero homomorphism from $P^\lam$ to $\theta_\gamma(P^\gamma)\cong L^\gamma$ for all $\lam\not=\gamma$.
Since $\theta_\gamma  f_{\gamma, \gamma}$ is a nonzero homomorphism with image contained in $\soc(P^{\gamma'})\cong L^{\gamma}$, we have $\theta_\gamma  f_{\gamma, \gamma}=c\theta_{\gamma}$ for some nonzero $c\in K$. Therefore
$$
\tau(\theta_{\gamma})=c^{-1}\tau(\theta_\gamma  f_{\gamma, \gamma})=c^{-1}\tau(\theta_\gamma  f)=c^{-1}(\theta_\gamma,f)_\tau\not= 0.
$$
Since $\tau$ is symmetric, we have $$
0\neq\tau(\theta_\gamma)=\tau\bigl(\id_{P^{\gamma'}} \theta_\gamma \id_{P^\gamma}\bigr)
=\tau\bigl(\id_{P^\gamma} \id_{P^{\gamma'}} \theta_\gamma\bigr).
$$
Therefore we have $\gamma'=\gamma$ for all $\gamma\in\Lam_0$.

Finally, for any $\lam,\mu\in\Lam_0$ with $\lam\neq\mu$ and $f\in\Hom_A(P^\lam,P^\mu)$, we have $$
\tau(f)=\tau\bigl(\id_{P^\mu}  f \id_{P^\lam}\bigr)=\tau\bigl(\id_{P^\lam} \id_{P^\mu}  f\bigr)=\tau(0)=0
$$
since $\tau$ is symmetric. This completes the proof of the lemma.
\end{proof}

For each $A$-module $M$, we define $\rad^0(M)=M$ and define the radical filtration on $M$ inductively by $\rad^i(M)=\rad(\rad^{i-1}(M))$. Note that $\rad^i(M)=(\rad A)^iM$, where $\rad(A)$ is the Jacobson radical of $A$.

Note that in Lemma \ref{keylemma0} we do not know whether we can choose $g$ to be $h$ or not in there satisfying $gf=\theta_\lam$ and $fh=\theta_{\mu}$. The following lemma ensures that the expectation holds if there exists a symmetrizing form $\tau$ on $B$ satisfying $\tau(\theta_\lam)=1$ for all $\lam\in\Lam_0$.

\begin{prop} \label{theataD} Assume that the map $'$ is an involution on $\Lam_0$ and there exists a symmetrizing form $\tau$ on $B$ such that $\tau(\theta_\lam)=1$ for all $\lam\in\Lam_0$. For any nonzero $f\in\Hom_A(P^\lam,P^\mu)$ with  $\lam,\mu\in\Lam_0$, there exists a homomorphism $g\in\Hom_A(P^\mu,P^\lam)$ such that
$fg=\theta_\mu$ and $gf=\theta_\lam$.
\end{prop}

\begin{proof} By Lemma \ref{keylem1}, we have $\lam'=\lam$ for all $\lam\in\Lam_0$. By Lemma \ref{keylemma0}, there is a homomorphism $h\in\Hom_A(P^\mu,P^\lam)$ satisfying $fh\not=0$. Let $s$ be the maximal integer such that there is a homomorphism $g\in\Hom_A(P^\mu,\rad^s(P^\lam))$ satisfying $fg\not=0$. We claim that $fg\in K^{\times}\theta_\mu$.

Suppose that $g\in\Hom_A(P^\mu,\rad^s(P^\lam))$ and $fg\notin K^{\times}\theta_\mu$. Then $L^\mu \cong\soc(P^{\mu})\subsetneq\im(fg)$. Applying Lemma \ref{keylemma0}, we can find $h\in\Hom_A(P^\mu,P^\mu)$ such that $fgh=\theta_\mu$. Since $\soc(P^{\mu})\subsetneq\im(fg)$, we can deduce that $h$ is not an isomorphism and hence $h$ is not injective (because every injective endomorphism of $P^\mu$ is automatically an automorphism). It follows that $\im(h)\subseteq\rad(P^\mu)$. By assumption that
$g\in\Hom_A(P^\mu,\rad^s(P^\lam))$, we have $\im(gh)\in \rad^{s+1}(P^\lam)$ and $f(gh)\neq 0$, we get a contradiction to the maximality of $s$. This proves our claim.

Now we may assume that $fg=\theta_\mu$. Note that $\tau(gf)=\tau(fg)=\tau(\theta_\mu)\neq 0$
by Lemma \ref{keylem1}. It follows that $gf\not=0$. By Lemma \ref{keylemma0} again we can find a homomorphism $h\in\Hom_A(P^\lam,P^\lam)$ such that $gfh=\theta_\lam$. It follows that $\tau(fhg)=\tau(gfh)=\tau(\theta_\lam)\neq 0$ by Lemma \ref{keylem1}. In particular, $fhg\neq 0$. We claim that $h$ is an isomorphism. Otherwise, $h$ is not injective (because every injective endomorphism of $P^\mu$ is automatically an automorphism) and hence $\im(h)\subseteq\rad(P^\lam)$.
It follows that for all $i\ge 0$,
$$
h(\rad^{i}(P^\lam))=h(\rad^i(A)P^\lam)=\rad(A)^i h(P^\lam)\subseteq\rad(A)^i\rad(P^\lam)=\rad^{i+1}(P^\lam).
$$
Then $hg(P^\mu)\subseteq h(\rad^s(P^\lam))\subseteq\rad^{s+1}(P^\lam)$. We get a contradiction to our assumption because $hg(P^\mu)\subseteq \rad^{s+1}(P^\lam)$ and $f(hg)\not=0$. This proves our claim. Therefore $h$ is an isomorphism and $gf\in K^\times\theta_\lam$. Since
$\tau(gf)=\tau(fg)$ and $\tau(\theta_\lam)=1$ for any $\lam\in\Lam_0$, it follows that $gf=\theta_\lam$.
\end{proof}
\bigskip

\section{Graded algebras}\label{Graded}
We are interested to find conditions ensuring that the canonical form $\tr$ attached to a given appropriate basis of the endomorphism algebra of any projective-injective module over any $\Z$-graded finite-dimensional algebra $A$ is a symmetrizing form. In this section, we give a sufficient condition, called an admissible condition (see Definition~\ref{HAdmissibleCond} below), for the appropriate basis of the endomorphism algebra $B=\End_{A}\Bigl(\bigoplus_{\lam\in\Lam_0}P^\lam\Bigr)$ (see (\ref{Hecke})) of the basic projective-injective module so that the canonical form $\tr$ attached to the basis is symmetric. For certain positively $\Z$-graded finite-dimensional algebras $A$,  $B$ is a symmetric algebra if and only if there exists an admissible $K$-basis of $B$. Moreover, the canonical form $\tr$ attached to the admissible basis is a symmetrizing form, see Corollary \ref{symm-eq} and Proposition \ref{keylem3} below.


For a field $K$, a graded $K$-vector space $M$ means a $\Z$-graded $K$-vector space $M=\bigoplus_{k\in\Z}M_k$ such that $M_d$ is a  finite-dimensional $K$-subspace of $M$ for each $d\in \Z$. For $d\in\Z$ and $v\in M_d$, $v$ is called a homogeneous element of degree $d$ and we write $\deg v=d$.
A graded $K$-algebra means a finite-dimensional associative unital $K$-algebra $R$ such that $R=\bigoplus_{d\in\Z}R_d$ is a graded $K$-vector space satisfying $R_dR_k\subset R_{d+k}$, for all $d,k\in\Z$. It follows that $1\in R_0$. A graded  $K$-algebra $R$ is called positively graded if $R=\bigoplus_{d\in\Z_+}R_d$. A graded (left) $R$-module is a graded  finite-dimensional  $K$-vector space $M=\bigoplus_{d\in\Z}M_d$ such that $M$ is an
$R$-module and $R_kM_d\subset M_{d+k}$, for all $d,k\in\Z$. For a graded  $K$-module $M$ and $k\in\Z$, let $M\<k\>$ be the graded $K$-module obtained by shifting the grading on~$M$ up by~$k$. That is,
$M\<k\>_d:=M_{d-k}$, for $d\in\Z$. For graded $R$-modules $M$ and $N$ and $d\in\Z$, $f\in\Hom_R(M,N)$ is called a homogeneous homomorphism of degree $d$ if $f(M_k)\subset N_{d+k}$ for all $k\in\Z$. Let  $\Hom_R(M,N)_d$ denote the subspace of $\Hom_R(M,N)$ consisting of homogeneous homomorphisms of degree $d$. Then $\Hom_R(M,N)$ forms a graded  $K$-vector space and $\Hom_R(M,N)=\bigoplus_{d\in\Z}\Hom_R(M,N)_d$. For graded $R$-modules $M$ and $N$, let
$$
\hom_R(M,N):=\Hom_R(M,N)_0.
$$
and let $M\simeq N$ denote that there is a homogeneous isomorphism of degree $0$ between $M$ and $N$. For a graded  $K$-algebra $R$, let $R\gmod$ denote the category of graded  finite-dimensional $R$-modules with homomorphisms of degree $0$.
The graded length of $M\in R\gmod$ is defined to be \begin{equation}\label{gr-length}
 b-a+1, \quad\text{for $M=\bigoplus_{i= a}^{b}M_i$ with $M_a\not=0\not=M_b$.}
\end{equation}
In particular, the graded length of $M$ is $b+1$ if $M=\bigoplus_{i=0}^{b}M_i$ with $M_0\not=0\not=M_b$.

The forgetful functor $R\gmod\rightarrow  R\mod$ is denoted by $\underline{\mathrm{For}}$. An $R$-module $M$ is called gradable if $M\cong \underline{\mathrm{For}}(N)$ for some $N\in R\gmod$. In this case, $M$ is also said to have a graded lift. For any modules $M, N\in R\gmod$ and $j,k\in\Z$, let $\sigma_{j,k}$ denote the isomorphism of $K$-vector spaces
 \begin{align}\label{gr-shift}
\sigma_{j,k}: \Hom_R(M,N) &\longrightarrow\Hom_R(M\<j\>,N\<k\>) \\
\nonumber  & f  \mapsto f[j,k]:=\sigma_{j,k}(f)
 \end{align}
such that $\underline{\mathrm{For}}( f[j,k])=\underline{\mathrm{For}}(f)$ for all $f\in \Hom_R(M,N)$. We also let
\begin{align}\label{gr-shift=}
f[j]:=f[j,j],\quad\text{for $f\in\Hom_R(M,N) $ and $j\in\Z$.}
 \end{align}
Note that
$$
\sigma_{j,k}(\Hom_R(M,N)_d)=\Hom_R(M\<j\>,N\<k\>)_{d-j+k}\qquad\text{for all $d\in \Z$.}
$$

We will adapt the notations and assumptions defined in \secref{End}. Recall that $A$ is a finite-dimensional algebra over an algebraically closed field $K$. In the section, we will further assume that $A$ is a positively graded $K$-algebra.
Since $A$ is positively graded, every simple module in $A\gmod$ concentrates in a fixed degree. Therefore every simple $A$-module is gradable. Recall that $\{L^\lam|\lam\in\Lam\}$ denotes a complete set of pairwise non-isomorphic simple $A$-modules. For each $\lam\in\Lam$, we fix a $\Z$-grading on $L^\lam$ by letting it concentrated in degree $0$. The resulting  graded simple $A$-module is also denoted by $L^\lam$. Then $\{L^\lam\<k\>|\lam\in\Lam,k\in\Z\}$ forms a complete set of pairwise non-isomorphic graded simple $A$-modules.

For each $\lam\in\Lam$, the projective cover $P^\lam$ of $L^\lam$ is gradable (see, for example, \cite[Corollary 3.4]{GordonGreen:GradedArtin}). Note that the natural homomorphism from $P^\lam$ to $L^\lam$ is of degree $0$.  Therefore we can write $P^\lam$ in the following form
\begin{equation}\label{deg-proj}
P^\lam=\bigoplus_{i=0}^{d_\lam}P^\lam_i\quad \text{such that $P^\lam_{d_\lam}\not=0$ for all $\lam\in\Lam$.}
\end{equation}
As a consequence, $B=\End_{A}\Bigl(\bigoplus_{\lam\in\Lam_0}P^\lam\Bigr)$ is a positively graded  algebra.

\begin{lem} \label{G-keylemma0}  Assume that the map $'$ is an involution on $\Lam_0$. For $\lam,\mu\in\Lam_0$ and a nonzero $f\in\Hom_A(P^\lam,P^{\mu})_j$, there exist $g\in\Hom_A(P^{\mu},P^{\lam'})_{d_{\lam'}-j}$ and
$h\in\Hom_A(P^{\mu'},P^\lam)_{d_\mu-j}$ such that $gf=\theta_{\lam}$ and $fh=\theta_{\mu'}$.
\end{lem}

\begin{proof}By Lemma~\ref{keylemma0}, there exists $g\in\Hom_A(P^{\mu},P^{\lam'})$ such that $gf=\theta_\lam$. Write $g=\sum_i g_i$, where $g_i\in \Hom_A(P^{\mu},P^{\lam'})_i$ for all $i$. Then we have $g_{d_{\lam'}-j}f=\theta_\lam$. A similar proof shows the existence of $h$.
\end{proof}

For each $\lam\in\Lam_0$, the homomorphism $\theta_\lam$ defined in (\ref{theta}) is clearly a homogeneous element in $B$ since $\soc(P^{\lam'})$ is a graded submodule of $P^{\lam'}$ \cite[Theorem 3.5]{GordonGreen:GradedArtin}. By \eqref{deg-proj}, the degree of $\theta_\lam$ in $B$ is
\begin{equation}\label{deg-theta}
\deg\theta_\lam=d_{\lam'}\quad\text{for every $\lam\in\Lam_0$.}
\end{equation}

\begin{dfn}\label{HAdmissibleCond} Let $A$ be a positively graded $K$-algebra such that $\soc(P^\lam)\simeq L^\lam\<d_\lam\>$ for each $\lam\in\Lam_0$ (in particular, $\lam'=\lam$ for each $\lam\in\Lam_0$). An appropriate basis $\Upsilon$ of $B=\End_{A}\bigl(\oplus_{\lam\in\Lam_0}P^\lam\bigr)$ consisting of homogeneous elements in $B$ is said to satisfy the admissible condition if for each $\lam, \mu\in\Lam_0$ and $j\in \Z$ with $\Upsilon\cap\Hom_A(P^\lam,P^\mu)_j\not=0$, then the following conditions hold:
\begin{enumerate}
\item[(i)] $d_\mu=d_\lam$, where $d_\gamma$ is defined in \eqref{deg-proj};
\item[(ii)]  $|\Upsilon\cap\Hom_A(P^\mu,P^\lam)_{d_\lam-j}|=|\Upsilon\cap\Hom_A(P^\lam,P^\mu)_j|$;
\item[(iii)] for any $f\in\Upsilon\cap\Hom_A(P^\lam,P^\mu)_j$, there is exactly one element $g\in\Upsilon\cap\Hom_A(P^\mu,P^\lam)_{d_\lam-j}$ such that $gf=c_{gf}\theta_\lam$, $fg=c_{fg}\theta_\mu$ for some $c_{fg}=c_{gf}\in K^{\times}$, and $hf=fh=0$ for all $h\in \Upsilon\cap\Hom_A(P^\mu,P^\lam)_{d_\lam-j}$ with $h\not=g$.
\end{enumerate}
A basis of $B$ is called admissible if it is an appropriate basis satisfying the admissible condition.
\end{dfn}

\begin{rem}\label{rem-Adm} The condition (i) means that every $P^\lam$ with $\lam\in\Lam_0$ in the same block has the same graded length (see the discussion above Proposition~\ref{keylem4}). The condition (ii) in the definition follows from the condition (iii). The condition (iii) implies that for all  $f\in\Hom_A(P^\lam,P^\mu)_j$, $g\in
\Hom_A(P^\mu,P^\lam)_{d_\lam-j}$ with $\lam, \mu\in\Lam_0$ and for all $j\in \Z$, we have
\begin{equation*}
fg=c\theta_\mu \quad\Leftrightarrow\quad gf=c\theta_\lam \quad\quad\text{for any fixed $c\in K^{\times}$}.
\end{equation*}
\end{rem}

The following proposition follows easily from Remark~\ref{rem-Adm} and the definition of the canonical form $\tr$ attached to an admissible basis defined in (\ref{tr0}).

\begin{prop} \label{keylem3} Assume that $A$ be a positively graded algebra with an admissible basis $\Upsilon$. If $\soc(P^\mu)\simeq L^\mu\<d_\mu\>$ for each $\mu\in\Lam_0$, then the canonical form $\tr$ attached to the admissible basis $\Upsilon$ is a symmetrizing form on $B$. In particular, $B$ is a symmetric algebra over $K$.
\end{prop}

\begin{prop} \label{homogekeylem2} Assume $A$ is a positively graded algebra and
$P^\lam_{d_\lam}=\soc(P^\lam)\cong L^\lam\<d_\lam\>$ for all $\lam\in\Lam_0$. If there exists a symmetrizing form $\tau$ on $B$, then there exists an admissible basis of $B$.
\end{prop}

\begin{proof} By assumption and Lemma \ref{keylem1}, we have $\lam'=\lam$ and $\tau(\theta_\lam)\neq 0$ for all $\lam\in\Lam_0$. We may choose $\theta_\lam$s such that $\tau(\theta_\lam)=1$ for all $\lam\in\Lam_0$. By assumption, we have $P^\lam=\bigoplus_{i=0}^{d_\lam}P^\lam_i$ such that $P^\lam_{d_\lam}=\soc(P^\lam)\cong L^\lam\< d_\lam\>$ for all $\lam\in\Lam_0$.

For $\lam, \mu\in\Lam_0$ and $k\in \Z$ satisfying $\Hom_A(P^\lam,P^\mu)_k\not=0$, we claim that $d_\lam= d_\mu$. Let $f\in \Hom_A(P^\lam,P^\mu)_k$ be a nonzero homomorphism. By Lemma~\ref{G-keylemma0},  there exist $g\in\Hom_A(P^{\mu},P^{\lam})_{d_\lam-k}$ and
$h\in\Hom_A(P^{\mu},P^\lam)_{d_\mu-k}$ such that $gf=\theta_\lam$ and $fh=\theta_{\mu}$. Therefore $\tau(fg)=\tau(gf)=\tau(\theta_\lam)\not=0$ and $\tau(hf)=\tau(fh)=\tau(\theta_\mu)\not=0$ by Lemma~\ref{keylem1}. Hence $fg\not=0\not= hf$, and
\begin{equation*}
 d_\lam=\deg(gf)=\deg(fg)\le d_\mu,\quad\text{and}\quad d_\mu=\deg(fh)=\deg(hf)\le d_\lam.
\end{equation*}
This implies $d_\lam= d_\mu$ and $\Hom_A(P^\mu,P^\lam)_{d_\lam-k}\not=0$.

Now we will show that there are bases of $\Hom_A(P^\lam,P^\mu)_k$ and $\Hom_A(P^\mu,P^\lam)_{d_\lam-k}$ satisfying the conditions (ii) and (iii) of Definition~\ref{HAdmissibleCond} for the case $\Hom_A(P^\mu,P^\lam)_{d_\lam-k}\not=\Hom_A(P^\lam,P^\mu)_k\not=0$. By Lemma~\ref{G-keylemma0}, there is a non-degenerate pairing from $\Hom_A(P^\lam,P^\mu)_k\times\Hom_A(P^\mu,P^\lam)_{d_\lam-k}$ to $K\theta_\mu$ defined by $(f, g)$ sending to $fg$. Then we have $\dim\,\Hom_A(P^\lam,P^\mu)_k=\dim\,\Hom_A(P^\mu,P^\lam)_{d_\lam-k}$, and a basis $\{f_1\cdots, f_m\}$ of $\Hom_A(P^\lam,P^\mu)_k$ and a basis $\{g_1\cdots, g_m\}$ of $\Hom_A(P^\mu,P^\lam)_{d_\lam-k}$ such that $f_ig_j=\delta_{ij}\theta_\mu$ for all $i,j$. Since $P^\lam_{d_\lam}=\soc(P^\lam)$ and the degree of $g_jf_i$ equals $d_\mu=d_\lam$ for all $i,j$, we have $g_jf_i\in K\theta_\lam$ for all $i,j$. Let $g_jf_i=c_{ji}\theta_\lam$. Then $$c_{ji}=\tau(c_{ji}\theta_\lam)=\tau(g_jf_i)=\tau(f_ig_j)=\tau(\delta_{ij}\theta_\mu)=\delta_{ij}.$$ Therefore $f_ig_j=\delta_{ij}\theta_\mu$ and $g_jf_i=\delta_{ij}\theta_\lam$ for all $i,j$.

Now we consider the case $\Hom_A(P^\lam,P^\mu)_k=\Hom_A(P^\mu,P^\lam)_{d_\lam-k}$ with $\Hom_A(P^\lam,P^\mu)_k\not=0$. Then we have $\lam=\mu$, $k={d_\lam-k}$ and a non-degenerate pairing from $\Hom_A(P^\lam,P^\lam)_k\times\Hom_A(P^\lam,P^\lam)_k$ to $K\theta_\lam$ defined by $(f, g)$ sending to $fg$. By some standard arguments, there is a basis $\{f_1\cdots, f_m\}$ of $\Hom_A(P^\lam,P^\lam)_j$ such that for each $j$ there is unique $j'$ satisfying $f_jf_{j'}=\theta_\lam$ and $f_jf_{i}=0$ for all $i\not=j'$. Using the same argument as above, we have $f_{j'}f_j=f_jf_{j'}=\theta_\lam$ and $f_if_j=f_jf_i=0$ for $i\not=j'$.

Taking the union of all bases obtained from above, we get an admissible basis of $B$.
\end{proof}

\begin{rem} If $A$ is a positively graded algebra such that $A_0$ is a semisimple $K$-algebra, then the assumption that $P^\lam_{d_\lam}=\soc(P^\lam)$ for all $\lam\in\Lam_0$ automatically holds.
\end{rem}

The following corollary is a consequence of Proposition~\ref{keylem3} and Proposition~\ref{homogekeylem2}.

\begin{cor}\label{symm-eq} Assume $A$ is a positively graded algebra such that
$P^\lam_{d_\lam}=\soc(P^\lam)\simeq L^\lam\<d_\lam\>$ for each $\lam\in\Lam_0$. Then $B$ is a symmetric algebra if and only if there exists an admissible $K$-basis of $B$.
\end{cor}

\begin{dfn} \label{HomoeneousTraceForm}Let $R$ be a graded  $K$-algebra. A form $\tau: R\rightarrow K$ is called homogeneous if $\tau$ is a homogeneous map, where $K$ is regarded as a graded  vector space concentrated in degree $0$. If $\tau: M\rightarrow K$ is a homogeneous form, then the associated bilinear form $(-,-)_{\tau}$ is also called a homogeneous bilinear form on $R$.
\end{dfn}

For a projective-injective $A$-module $Q$, $Q$ is gradable and $Q\simeq\bigoplus_{\mu\in\Lam_0}(P^{\mu})^{\oplus k_\mu}$ for some $k_\mu\in\Z_+$ such that not all $k_\mu$s are zero.

\begin{cor} \label{kmu} Assume that $A$ is a positively graded algebra and $P^\lam_{d_\lam}=\soc(P^\lam)\simeq L^\lam\<d_\lam\>$ for each $\lam\in\Lam_0$.
Let $Q$ be any projective-injective $A$-module. If $B$ is a symmetric algebra over $K$, then  $\End_{A}(Q)$ is a symmetric algebra over $K$. Moreover, if
the projective-injective $A$-module $Q=\bigoplus_{\mu\in\Lam_0}(P^{\mu})^{\oplus k_\mu}$ such that $d_\mu=d$ for all $k_\mu\not=0$, then there is a homogeneous symmetrizing form on $\End_{A}(Q)$ of degree $-d$.
\end{cor}

\begin{proof} By assumption and Lemma \ref{keylem1}, we have $\lam'=\lam$ and $\tau(\theta_\lam)\neq 0$ for all $\lam\in\Lam_0$. By Proposition~\ref{homogekeylem2} and Proposition~\ref{keylem3}, there exists an admissible basis of $B$ and $\tr$ is a symmetrizing form on $B$.

 Let $Q=\bigoplus_{\mu\in\Lam_0}(P^{\mu})^{\oplus k_\mu}$ for some $k_\mu\in\Z_+$
 and let $\widehat{B}:=\End_{A}\bigl(Q\bigr)$. Since $P^{\mu}$ is a  graded $A$-module for each $\mu\in \Lam_0$, $\widehat{B}$ is a positively graded $K$-algebra. For $\mu\in \Lam_0$ and  $s\in \Z_{> 0}$ such that $1\le s\le k_\mu$, let $(P^{\mu})^{(s)}$ denote the $s$-th component of $(P^{\mu})^{\oplus k_\mu}$. We regard that $(P^{\mu})^{(s)}\subseteq (P^{\mu})^{\oplus k_\mu}\subseteq Q$ and $\Hom_A\bigl((P^{\mu})^{(s)}, (P^{\lam})^{(t)}\bigr)\subseteq \widehat{B}$ in the natural way for $\mu, \lam\in \Lam_0$, $1\le s\le k_\mu$ and $1\le t\le k_\lam$. For $\mu, \lam\in \Lam_0$, $1\le s\le k_\mu$ and $1\le t\le k_\lam$, there is an isomorphism of $K$-vector space from $\Hom_A\bigl(P^{\mu}, P^{\lam}\bigr)$ to $\Hom_A\bigl((P^{\mu})^{(s)}, (P^{\lam})^{(t)}\bigr)$ by sending $f$ to $f^{(s,t)}$ defined in the obvious way. Define a form on $\widehat{B}$ determined by
\begin{equation*}
\widehat{\tr}(f):=\begin{cases} \tr(g), &\text{if $f\in \Hom_A\bigl((P^{\mu})^{(s)}, (P^{\lam})^{(t)}\bigr)$ such that $f=g^{(s,t)}$ with $s=t$;}\\
0, &\text{if $f\in \Hom_A\bigl((P^{\mu})^{(s)}, (P^{\lam})^{(t)}\bigr)$ with $s\not=t$.}
\end{cases}
\end{equation*}
It is clear that $\widehat{\tr}$ is a symmetrizing form on $\widehat{B}$ by Proposition~\ref{homogekeylem2}, Proposition~\ref{keylem3} and Remark~\ref{rem-Adm}. Finally, it is obvious that $\widehat{\tr}$ is a homogeneous linear map on $\End_{A}(Q)$ of degree $d$ if $d_\mu=d$ for all $k_\mu\not=0$.
\end{proof}


For the rest of this section, we assume that $A$ is a positively graded $K$-algebra equipped with a homogeneous anti-involution $\star$ of degree $0$ satisfying  $(L^\lam)^{\circledast}\simeq L^\lam$ (defined below) for each $\lam\in\Lam$.
The dual of the graded $A$-module $M$ is the graded $A$-module
$$
M^\circledast =\bigoplus_{j\in\Z}M^\circledast_j , \qquad \text{where $M^\circledast_j:=\Hom_K(M_{-j},K)$,}
$$
and the action of $A$ on~$M^\circledast$ is given by $(af)(m)=f(a^\star m)$ for all $f\in M^\circledast$, $a\in A$ and~$m\in M$. It is clear that
\begin{equation}\label{gr-D}
(M^\circledast)^\circledast\simeq M\quad \text{and }\quad  (M\<k\>)^\circledast\simeq M^\circledast\<-k\>\quad \text{$\forall\,M\in A\gmod, k\in\Z$}.
\end{equation}
Also $\circledast$ gives an equivalence of categories. Recall the graded simple module $L^\lam\in A\gmod$ is concentrated in degree zero. Therefore $(L^\lam)^{\circledast}$ is a graded simple module concentrating in degree zero for every $\lam\in\Lam$. The assumptions say
 \begin{equation}\label{re-D-L}
  (L^\lam)^{\circledast}\simeq L^\lam\quad \text{for each $\lam\in\Lam$}.
  \end{equation}
  and hence $L^\lam\<k\>^{\circledast}\simeq L^\lam\<-k\>$ for each $\lam\in\Lam$ and $k\in\Z$.

 \begin{lem}\label{degP} Assume that $A$ is a positively graded $K$-algebra equipped with a homogeneous anti-involution $\star$ of degree $0$ satisfying  $(L^\lam)^{\circledast}\simeq L^\lam$ for each $\lam\in\Lam$. Assume further that $L^\lam\<d_\lam\>\simeq\soc(P^\lam)\subseteq P^\lam_{d_\lam}$ for each $\lam\in\Lam_0$. Then we have
 \begin{equation}\label{gr-dual-proj}
  (P^{\lam})^{\circledast}\simeq P^\lam\<-d_\lam\> \quad\text{for each $\lam\in\Lam_0$.}
\end{equation}
\end{lem}

\begin{proof} $P_{d_\lam}^\lam$ is a graded $A$-submodule of $P^\lam$ since $A$ is positively graded. Therefore  $P_{d_\lam}^\lam$ contains the simple socle $\soc(P^\lam)$ of $P^\lam$.
 Now we show \eqref{gr-dual-proj}. There is an injective map $g\in\hom_A(L^\lam\<d_\lam\>,P^{\lam})$ since $\soc (P^\lam)\subseteq P^\lam_{d_\lam}$. Therefore there is a surjective map $h\in \hom_A\bigl((P^{\lam})^\circledast,L^{\lam}\<-d_\lam\>\bigr)$ because $(L^{\lam}\<d_\lam\>)^\circledast\simeq L^{\lam}\<-d_\lam\>$. Now $(P^{\lam})^{\circledast}\simeq P^\lam\<-d_\lam\>$ follows from the fact that $\For((P^{\lam})^{\circledast})$ is a projective cover of $L^\lam$ in $A\mod$.
\end{proof}

\begin{dfn}\label{hwCat}\text{(\cite{CPS:HomDual, Maksimau:QuiverSchur})} Let $A$ be a finite-dimensional positively graded  algebra over an algebraically closed field $K$. Let $\mathcal{C}:=A\gmod$ denote the category of graded finite-dimensional $A$-modules. Let $\star$ be a homogeneous anti-involution of degree zero of $A$ which induces a graded duality functor $\circledast$ on $\mathcal{C}$. Let $(\Lam,\leq)$ be a finite poset and let $\Delta:=\{\Delta^\lam|\lam\in\Lam\}$ be a family of objects in $\mathcal{C}$. The pair $(\mathcal{C},\Delta)$ is called a $\Z$-graded highest weight category with a duality functor $\circledast$ if \begin{enumerate}
\item[(i)] for each $\lam\in\Lam$, $\Delta^\lam$ has a unique simple head $L^\lam$ satisfying $(L^\lam)^{\circledast}\cong L^\lam$, and $\{L^\lam\<k\>|\lam\in\Lam,k\in\Z\}$ is a complete set of pairwise non-isomorphic simple modules in $\mathcal{C}$;
\item[(ii)] $\Hom_{\mathcal{C}}(\Delta^\lam,M)=0$ for each $\lam\in\Lam$ implies that $M=0$;
\item[(iii)] for each $\lam\in\Lam$, there is an indecomposable graded projective module $P^\lam$ in $\mathcal{C}$ and a degree $0$ surjective homomorphism $f: P^\lam\twoheadrightarrow\Delta^\lam$ such that $\Ker(f)$ has a finite filtration whose successive quotient are objects of the form $\Delta^\mu\<k\>$ with $\mu>\lam$ and $k\in\Z$;
\item[(iv)] for each $\lam\in\Lam$, we have $\End_{\mathcal{C}}(\Delta^\lam)\cong K$;
\item[(v)] for each $\lam,\mu\in\Lam$ such that $\Hom_{\mathcal{C}}(\Delta^\lam,\Delta^\mu)\neq 0$, we have $\lam\leq\mu$.
\end{enumerate}
\end{dfn}

Let $v$ be an indeterminate over $\Z$. For $M\in A\gmod$ and a graded simple $A$-module $L$, the graded dimension of $M$ is the Laurent polynomial
\begin{equation}\label{E:GrDim} \dim_v M:=\sum_{d\in\Z}(\dim_K M_d)\,v^d ,
\end{equation}
and
the graded multiplicity of $L$ in $M$ is the Laurent polynomial
\begin{equation}\label{E:multiplicity}
[M:L]_v := \sum_{k\in\Z}[M:L\<k\>]v^k.
\end{equation}
For any $\lam,\mu\in\Lam$, the graded Cartan matrix $c_{\lam,\mu}(v)$ is defined by
 $$
c_{\lam,\mu}(v):=\dim_v\Hom_{A}(P^\mu,P^\lam)=\sum_{k\in\Z}[P^\lam:L^\mu\<k\>]v^k .
$$

Every finite-dimensional $K$-algebra $R$ has a unique decomposition into a direct sum of indecomposable blocks. A block of $R$ means an indecomposable two-sided ideal of $R$.  It is well known that the equivalence relation on the simple $R$-modules induced from the block decomposition of $R$, where two simples are equivalent if they belong to the same block, coincides
with the linkage classes of simple $R$-modules, where the equivalence relation is
generated by $L^\lam\sim L^\mu$ if $\Ext^1(L^\lam, L^\mu)\not= 0$ or $\Ext^1(L^\mu, L^\lam)\not= 0$, where $L^\lam$ and $L^\mu$ are simple $R$-modules. Also the linkage classes of simple $R$-modules coincides with the equivalence relation generated by $L^\lam\sim L^\mu$ if $\Hom_R(P^\lam, P^\mu)\not= 0$ or $\Hom_R(P^\mu, P^\lam)\not= 0$, where $P^\lam$ and $P^\mu$ are the projective covers of the simple $R$-modules $L^\lam$ and $L^\mu$, respectively. We also let $\Hom_{-R}(U,V)$ denote the $K$-vector space of $R$-homomorphisms from the right $R$-module $U$ to the right $R$-module $V$.

\begin{prop} \label{keylem4} Let $A$ be a positively graded algebra equipped with a homogeneous anti-involution $\star$ of degree $0$ such that $\soc(P^\lam)\simeq L^\lam\<d_\lam\>$ for all $\lam\in\Lam_0$.
Assume that either $(A,\star)$ is a $\Z$-graded cellular algebra in the sense of \cite{HuMathas:GradedCellular} or $(A\gmod,\circledast)$ is a $\Z$-graded highest weight category with a duality functor.
If there is a homogeneous idempotent $e\in A$ of degree zero satisfying $Ae\simeq\oplus_{\lam\in\Lam_0}P^\lam$ and $e^\star=e$, and there is a double centralizer property between $A$ and $B^{\mathrm{o}}:=\Bigl(\End_{A}(Ae)\Bigr)^{\mathrm{op}}$ on $Ae$ (i.e., the canonical map
$A\rightarrow\End_{-{B}^{\mathrm{o}}}(Ae)$ is an isomorphism),
 then $d_\lam=d_\mu$ for any $\lam,\mu\in\Lam_0$ such that $L^\lam$ and $L^\mu$ belong to the same block.
\end{prop}

\begin{proof} We will follow the definitions from \cite{HuMathas:WeylSocle, Maksimau:QuiverSchur} closely in the proof of the proposition.  Note that $B^{\mathrm{o}}\simeq eAe$ since $e$ is a homogeneous idempotent of degree zero. Let $B'=eAe$, and let $A=\bigoplus^m_{i=1}A_i$ and $B'=\bigoplus^n_{j=1}B'_j$ be the decompositions of $A$ and $B'$ into indecomposable blocks, respectively. Let $e_1,\cdots,e_m$ be the block idempotents of $A$.
Let $F: A\mod\rightarrow {B}^{\mathrm{o}}\mod$ be the Schur functor defined by: $F(M):=eM$, $F(f): ev\mapsto ef(v)$, for all $M,N\in A\mod, f\in\Hom_A(M,N), v\in M$. We first show that $F$ induces a bijection between the indecomposable blocks of $A$ and $B'$. The proof is the same as \cite[Corollary 2.19]{HuMathas:WeylSocle}. For completeness, we include it here.

First we will show $m=n$. Let $Q:=Ae$ and $Q_j:=e_jQ$, for $1\le j\le m$. Since $A\simeq\End_{-B'}(Ae)$, $Ae$ is a faithful left $A$-module. Then $Q_r\ne0$ for $1\le r\le m$, so $Q=\bigoplus^m_{i=1}Q_i$ is the decomposition of~$Q$ into its block components. In particular,
$\Hom_A(Q_r,Q_s)=0$ if $r\ne s$.
  Therefore, ${B}^{\mathrm{o}}=\End_A(Q)^{\mathrm{op}}\simeq\bigoplus^m_{i=1}\End_A(Q_i)^{\mathrm{op}}$ is a decomposition of~${B}^{\mathrm{o}}$ into (not necessarily indecomposable) blocks.  In particular, $n\ge m$.

On the other hand, the assumption that $e^\star=e$ and the isomorphism $A\simeq\End_{-B'}(Ae)$ imply that the canonical map \begin{equation}\label{psip}
A^{\mathrm{op}}\rightarrow\End_{B'}(eA) .
\end{equation}
is an isomorphism too, which implies the Schur functor $F$ is fully faithful on projectives. This means
 \begin{equation}\label{psip2}
\Hom_A(P^\lam,P^\mu)\simeq\Hom_{B'}(eP^\lam, eP^\mu) \quad\text{ for all $\lam,\mu\in\Lam$.}
\end{equation}
In particular, $Y^\lam:=eP^\lam$ is an indecomposable $B'$-module for all $\lam\in\Lam$. Note that $\{F(L^\lam)\,|\, \lam\in\Lam_0\}$ is a complete list of non-isomorphic simple $B'$-modules
 and $Y^\lam$ is the projective cover of simple $B'$-module $F(L^\lam)$ for all $\lam\in\Lam_0$.
 For $\lam,\mu \in\Lam_0$ such that $L^\lam$ and $L^\mu$ belong to the same block of $A$, there exists a sequence $\lam=\gamma_1, \gamma_2,\cdots,\gamma_l=\mu$ of elements in $\Lam$ such that $\Hom_A(P^{\gamma_{j}},P^{\gamma_{j+1}})\not=0$ or $\Hom_A(P^{\gamma_{j+1}},P^{\gamma_{j}})\not=0$ for all $1\leq j\leq l-1$ from the discussion above this proposition.
Since $Y^\gamma$ are indecomposable $B'$-modules for all $\gamma\in\Lam$, $Y^{\gamma_{j}}$ and $Y^{\gamma_{j+1}}$ are in the same block of $B'$ for all $1\leq j\leq l-1$ by \eqref{psip2}. Therefore $F(L^\lam)$ and $F(L^\mu)$ belong to the same block of $B'$. This implies $n\le m$ and hence $m=n$. Also we obtain that the functor $F$ induces a bijection between the indecomposable blocks of~$A$ and~$B'$.

Let $\lam,\mu\in\Lam_0$ such that $P^\lam$ and $P^\mu$ belong to the same block of $A$. Now we are going to show that $d_\lam=d_\mu$. From above, the projective modules $Y^\lam$ and $Y^\mu$ belong to the same block of ~$B'$.
From the discussion above this proposition, there exists a sequence $\lam=\gamma_1, \gamma_2,\cdots,\gamma_l=\mu$ of elements in $\Lam_0$ such that $\Hom_{B'}(Y^{\gamma_{j}},Y^{\gamma_{j+1}})\not=0$ or $\Hom_{B'}(Y^{\gamma_{j+1}},Y^{\gamma_{j}})\not=0$ for all $1\leq j\leq l-1$ since $\{Y^\lam|\lam\in\Lam_0\}$ is a complete set of non-isomorphic indecomposable projective $B^\mathrm{o}$-modules.
By \eqref{psip2}, we have $\Hom_A(P^{\gamma_{j}},P^{\gamma_{j+1}})\not=0$ or $\Hom_A(P^{\gamma_{j+1}},P^{\gamma_{j}})\not=0$ for all $1\leq j\leq l-1$. To show that $d_\lam=d_\mu$, it suffices to show that $d_{\lam_j}=d_{\lam_{j-1}}$ for all $1\leq j\leq l-1$. Therefore it is enough to show $d_\lam=d_\mu$ when $\Hom_{A}(P^\mu,P^\lam)\neq 0$ and $\lam,\mu\in\Lam_0$.

In the case that $(A,\star)$ is a $\Z$-graded cellular algebra in the sense of \cite{HuMathas:GradedCellular}, by \cite[Theorem 2.17]{HuMathas:GradedCellular}, we have
$$
c_{\lam,\mu}(v)=\sum_{\gamma\in\Lam}[C^\gamma:L^\lam]_v[C^\gamma:L^\mu]_v=c_{\mu,\lam}(v),\quad\text{for all $\mu,\lam\in\Lam$,}
$$
where $C^\gamma$ is the $\Z$-graded cell module associated to $\gamma\in\Lambda$ and $[C^\gamma:L^\lam]_v$ defined in \eqref{E:multiplicity} is the graded multiplicities of $L^\lam$ in the graded module $C^\gamma$.

Now we assume that  $(A\gmod,\circledast)$ is a $\Z$-graded highest weight category with a duality functor $\circledast$ induced from $\star$. Then we have \eqref{gr-D} and \eqref{re-D-L}. By \cite[Corollary 2.16]{Maksimau:QuiverSchur}, we have that $$
(P^\lam:\Delta^\gamma)_v=[\Delta^\gamma:L^\lam]_v,\quad\forall\,\gamma,\lam\in\Lam ,
$$
where $(P^\lam:\Delta^\gamma)_v$ is defined similar to (\ref{E:multiplicity}) to be the graded filtration multiplicities of $\Delta^\gamma$ in $P^\lam$. It follows that
$$
c_{\lam,\mu}(v)=\sum_{\gamma}(P^\lam:\Delta^\gamma)_v[\Delta^\gamma:L^\mu]_v
=\sum_{\gamma}[\Delta^\gamma:L^\lam]_v[\Delta^\gamma:L^\mu]_v=c_{\mu,\lam}(v),\quad\text{for all $\mu,\lam\in\Lam$.}
$$

Therefore, in both cases, we have $c_{\lam,\mu}(v)=c_{\mu,\lam}(v)$ for all $\mu,\lam\in\Lam$. On the other hand, we have by \eqref{gr-dual-proj} that $$
\begin{aligned}
c_{\lam,\mu}(v)&=\dim_v\Hom_{A}(P^\mu,P^\lam)=\dim_v\Hom_{A}\bigl((P^\lam)^{\circledast},(P^\mu)^{\circledast}\bigr)
=\dim_v\Hom_{A}(P^\lam\<-d_\lam\>,P^\mu\<-d_\mu\>)\\
&=v^{-d_\mu+d_\lam}\dim_v\Hom_{A}(P^\lam,P^\mu)=v^{-d_\mu+d_\lam}c_{\mu,\lam}(v)=v^{-d_\mu+d_\lam}c_{\lam,\mu}(v),
\end{aligned}
$$
By assumption that $c_{\lam,\mu}(v)\neq 0$, thus we can deduce that $v^{d_\mu-d_\lam}=1$ and hence $d_\mu=d_\lam$ as required.\end{proof}


\bigskip

\section{Parabolic BGG category and the proof of Theorem \ref{mainthm1}}
In this section, we shall study the endomorphism algebra of any projective-injective module in parabolic BGG category $\O^\p$ over the field $\C$ of complex numbers. After recalling some  preliminarily results on parabolic BGG category $\O^\p$, we give a proof of Theorem \ref{mainthm1}. One of the key ingredient of the proof of Theorem \ref{mainthm1} is Proposition \ref{keyprop}.

Let $\g$ be a complex semisimple Lie algebra with a fixed Borel subalgebra $\b$ containing the Cartan subalgebra $\h$ and let $\O$ denote the corresponding Bernstein-Gelfand-Gelfand (BGG)  category \cite{Humphreys:BGG}. Let $\Phi$ be the root system of $\g$ relative to $\h$,  $\Delta$ the set of simple roots in $\Phi$ corresponding to $\b$ and $\Phi^+$ the set of positive roots in $\Phi$. Let $W$ be the Weyl group of $\g$ attached to the root system $\Phi$. An element $\lam\in\h^\ast$ is called a weight of $\g$. For any weight $\lam\in\h^\ast$ and $w\in W$, we define $w\cdot\lam:=w(\lam+\rho)-\rho$, where $\rho$ is the half-sum of positive roots in $\Phi$. For any $\lam\in\h^\ast$, let $L(\lam)\in\O$ denote the irreducible highest weight module with highest weight $\lam$.
A weight $\lam\in\h^\ast$ is said to be integral (resp. dominant) if $\<\lam,\alpha^\vee\>\in\Z$ (resp. $\<\lam+\rho,\alpha^\vee\>\geq 0$) for any $\alpha\in\Delta$, where $\alpha^\vee$ denotes the coroot of $\alpha$.
A dominant integral weight $\lam$ is called regular if $\<\lam+\rho,\alpha^\vee\>> 0$ for any $\alpha\in\Delta$, otherwise $\lam$ is called singular. Let $\Lambda$ denote the set of integral weights of $\g$.

Let $I\subset\Delta$ be a subset of $\Delta$ which defines a root system $\Phi_I\subset \Phi$ with positive roots $\Phi_I^+\subset\Phi^+$ and negative roots $\Phi_I^-\subset\Phi^-$ and let $W_I$ be the Weyl group generated by all $s_\alpha$ with $\alpha\in I$. Associated with the root system $\Phi_I$ we have the standard parabolic subalgebra $\p:=\p_I\supseteq\b$ which has a Levi decomposition $\p_I=\mathfrak{l}_I\oplus \mathfrak{u}_I$, where $\mathfrak{l}_I:=\h\bigoplus\bigoplus_{\alpha\in\Phi_I}\g_\alpha$ is the Levi subalgebra of $\p$ and $\mathfrak{u}_I:=\bigoplus_{\alpha\in\Phi^+\setminus\Phi_I^+}\g_\alpha$ is the nilradical of $\p$. We define $\Lambda_\p^{+}:=\{\lam\in\Lambda|\<\lam,\alpha^\vee\>\geq 0,\,\text{for all $\alpha\in I$}\}$.
Let $\O^\p$ denote the corresponding parabolic BGG category \cite{Ro}.
For each $\lam\in\Lambda$, $L(\lam)$ lies in $\O^\p$ if and only if $\lam\in\Lambda_\p^+$. For each $\lam\in\Lambda$, let $\O_\lam$ denote the  subcategory of $\O$ whose composition factors are all of the form $L(w\cdot\lam)$ for some $w\in W$ and let $\O_{\lam}^{\p}:=\O_\lam\bigcap\O^\p$. For $\lam\in\Lambda_\p^+$, let $\Delta(\lam)$ denote the parabolic Verma module with highest weight $\lam$ and $P(\lam)$ denote the projective cover of $L(\lam)$ in $\O^\p$.

 For any dominant  integral weight $\psi$, let $W_\psi:=\{w\in W|w\cdot\psi=\psi\}$ denote the stabiliser of $\psi$ in the Weyl group $W$, let $W^\psi$ denote the set of maximal length left coset representatives of $W_\psi$ in $W$ and let
 $$
 {{\mathring{W}}^\psi}:=\{w\in W^\psi\,|\, w\cdot\psi\in\Lam_\p^+\}.
 $$
 The simple modules $L(\mu)$ in $\O_\psi^\p$ are parameterized by $\mu \in \WLp\cdot \psi$.
 Let $(-)^\ast$ denote the dual functor on the BGG category $\O$. The dual functor $(-)^\ast$ on $\O$ descending to the parabolic BGG category $\O^\p$ is also denoted by $(-)^\ast$ (see, for example, \cite[Section 3.2]{Humphreys:BGG}, in which $(-)^\ast$ is denoted by ``$(-)^\vee$"). Recall that a module $M$ is called self-dual if $M^\ast$ is isomorphic to $M$. Note that $L(\mu)^*\cong L(\mu)$ for all $\mu\in \Lam_\p^+$. That is, every simple module in $\O^\p$ is self-dual.

Let $\lam$ denote a fixed dominant integral weight. Let $T_0^\lam: \O_0\rightarrow\O_\lam$ and $T_\lam^0: \O_\lam\rightarrow\O_0$ be the two translation functors defined by $T_0^\lam(-):=\pr_\lam(E\otimes -)$ and $T^0_\lam(-):=\pr_0(F\otimes -)$ \cite{Jantz1} (see also
\cite[Chapter 7]{Humphreys:BGG}, \cite{Backelin:Koszul}), where $E$ is a finite-dimensional irreducible $\g$-module with extremal weight $\lam$, $\pr_\lam$ is the projection from $\O$ onto the block $\O_\lam$,  $F$ is a finite-dimensional irreducible $\g$-module with extremal weight $-\lam$ and $\pr_0$ is the projection from $\O$ onto the block $\O_0$. The functors $T_0^\lam$ and $T^0_\lam$ descend to the functors $T_0^\lam: \O_0^\p\rightarrow\O_\lam^\p$ and $T_\lam^0: \O_\lam^\p\rightarrow\O_0^\p$ on the parabolic categories. Here and after, the functors $T_0^\lam$ and $T^0_\lam$ stand for the functors on the parabolic categories.
The following lemma is a collection of some well-known results of the translation functors on the parabolic categories (see, for example, \cite[Lemma 2.5, 2.6]{Backelin:Koszul}, \cite[Pages 130, 140, 143, 186, 192]{Humphreys:BGG}).

\begin{lem} \label{5properties} For a dominant integral weight $\lam$, we have the following:
\begin{enumerate}
  \item[(i)] $T_0^\lam: \O_0^\p\rightarrow\O_\lam^\p$ and $T_\lam^0: \O_\lam^\p\rightarrow\O_0^\p$ are exact functors adjoint to each other. Moreover, for any $M,N\in\O^\p_\lam$ and $f\in\Hom_{\O}(M,N)$, we have
$$
T_0^\lam   T^0_\lam(M)\cong\underbrace{M\oplus\cdots\oplus M}_{\text{$|W_\lam|$ copies}},\quad
T_0^\lam   T^0_\lam(f)\cong\underbrace{f\oplus\cdots\oplus f}_{\text{$|W_\lam|$ copies}}.
$$
  \item[(ii)] For $w\in W$ with $w\cdot 0\in \Lambda_\p^+$, we have
  $$
  T_0^\lam(L(w\cdot 0))
  \cong\begin{cases} L(w\cdot\lam),&\text{ if $w\in \WLl$};\\
0,&\text{ if $w\notin  \WLl$}.\end{cases}
$$
  \item[(iii)] $T_0^\lam$ and $T^0_\lam$ send projectives to projectives. Moreover, if $w\in\WLl$, then $T^0_\lam(P(w\cdot \lam))=P(w\cdot 0)$.\\
\item[(iv)] $T_0^\lam$ and $T^0_\lam$ commute with the dual functor $(-)^\ast$.
\end{enumerate}
\end{lem}

\begin{lem}\label{tlam0simple} For $w\in \WLl$, $T_\lam^0(L(w\cdot\lam))$ is self-dual with a simple head isomorphic to $L(w\cdot 0)$ and a simple socle isomorphic to $L(w\cdot 0)$. Moreover, we have
 $$
 [T_\lam^0(L(w\cdot\lam)):L(y\cdot 0)]=\delta_{y,w}|W_\lam|, \quad \text{for any $y\in {\mathring{W}}^\lam$}.
 $$
\end{lem}

\begin{proof} By Lemma~\ref{5properties}~(i)~and~(ii), for any $y\in W$ with $y\cdot 0\in \Lambda_\p^+$, we have,
\begin{align*}
 \Hom_\O(T_\lam^0(L(w\cdot\lam)), L(y\cdot 0))&\cong\Hom_{\O}(L(w\cdot\lam),T_0^\lam(L(y\cdot 0)))\\
 & \cong\begin{cases} \Hom_{\O}(L(w\cdot\lam),L(y\cdot\lam))\cong\delta_{y,w}\C,&\text{ if $y\in \WLl$};\\
0,&\text{ if $y\notin \WLl$},\end{cases}\\
 & = \delta_{y,w}\C.
\end{align*}
This proves that $T_\lam^0(L(w\cdot\lam))$ has a simple head isomorphic to $L(w\cdot 0)$. Since $\bigl(T_\lam^0(L(w\cdot\lam))\bigr)^{\ast}\cong T_\lam^0(L(w\cdot\lam)^{\ast})\cong T_\lam^0(L(w\cdot\lam))$, $T_\lam^0(L(w\cdot\lam))$ is self-dual with a simple socle isomorphic to $L(w\cdot 0)$.

Finally, by Lemma~\ref{5properties}~(i) and (ii), for any $y\in \WLl$ satisfying $[T_\lam^0(L(w\cdot\lam)):L(y\cdot 0)]\neq 0$, we have $L(y\cdot \lam)\cong T_0^\lam(L(y\cdot 0))$ and hence
$$
[T_0^\lam(T_\lam^0(L(w\cdot\lam)): L(y\cdot \lam)]
=[T_\lam^0(L(w\cdot\lam):L(y\cdot 0)] .
$$

By Lemma~\ref{5properties}~(i) again, we get that
$y=w$ and $[T_\lam^0(L(w\cdot\lam)):L(w\cdot 0)]=|W_\lam|$.
\end{proof}

\begin{rem}
A graded version of Lemma~\ref{tlam0simple} is given in Lemma~\ref{keylem41} below.
\end{rem}

\begin{dfn} \label{socular1} For $\gamma\in\Lambda_\p^+$, $\lam$ is said to be socular if $L(\gamma)$ lies in the socle of a parabolic Verma module $\Delta(\mu)$ for some $\mu\in\Lam_\p^+$.
\end{dfn}

\begin{lem} \label{socular2} \text{(\cite[Addendum]{Irv1})} For $\gamma\in\Lambda_\p^+$, $P(\gamma)$ is injective if and only if $\gamma$ is socular. In this case, $P(\gamma)$ is a tilting module and in particular it is self-dual.
\end{lem}

Lemma~\ref{socular2} implies that each $\O_\psi^\p$ contains a projective-injective module for any dominant integral weight $\psi$.

\begin{lem}\label{3properties} Let $w\in \WLl$. If $w\cdot\lam$ is socular, then $w\cdot 0$ is socular too.
\end{lem}

\begin{proof} We first show that $w\cdot 0\in\Lam_\p^+$. To this end, it suffices to show that $w^{-1}(\alpha)\in\Phi^+$ for any $\alpha\in I$.
By assumption, we have $w\cdot\lam\in\Lam_\p^+$ and hence $$
\<\lam+\rho, w^{-1}(\alpha)^\vee\>=\<\lam+\rho, w^{-1}(\alpha^\vee)\>=\<w(\lam+\rho),\alpha^\vee\>\geq \<\rho,\alpha^\vee\>=1,\quad\,\forall\,\alpha\in I .
$$
Since $\lam$ is dominant, we can deduce from the above inequality that $w^{-1}(\alpha)\in\Phi^+$ as required. This proves that $w\cdot 0\in\Lam_\p^+$.

By assumption that $w\cdot\lam$ is socular, it follows that $P(w\cdot\lam)$ is self-dual and hence $P(w\cdot0)\cong T_\lam^0(P(w\cdot\lam))$ is also self-dual. Therefore $w\cdot 0$ is a socular weight.
\end{proof}

\begin{dfn} \label{2definitions} For a dominant integral weight $\psi$, we set
\begin{equation}\label{2defn}
 \Lam_0^\psi:=\bigl\{\mu\in\WLp\cdot\psi\bigm|\text{$\mu$ is socular}\bigr\}\quad \text{and}\quad \check{W}^\psi:=\{w\in \WLp \bigm|\,  \text{  $w\cdot\psi$ is socular}\}.
\end{equation}
\end{dfn}

 It is clear that the map $w\mapsto w\cdot\psi$ defines a bijection between $\check{W}^\psi$ and $\Lam_0^\psi$. Note that $\check{W}^0=\{w\in W \bigm| w\cdot 0\in \Lam_\p^+\text{ and  $w\cdot 0$ is socular}\}$.
It follows from Lemma \ref{3properties} that
\begin{equation}\label{2inclusion}
 \check{W}^\lam\subset\check{W}^0.
 \end{equation}

\begin{dfn}\label{KoszulBasic} For a dominant integral weight $\psi$, we define the basic algebra of the category $\O_\psi^\p$ by $$
A^\p(\psi):=\biggl(\End_{\O^\p}\Bigl(\oplus_{\mu\in {{\mathring{W}}^\psi}\cdot\psi}P(\mu)\Bigr)\biggr)^{\mathrm{op}}.
$$
\end{dfn}

The functor $\mathcal{F}^\psi:=\Hom_{\O^p}\bigl(\oplus_{\mu\in {{\mathring{W}}^\psi\cdot\psi}}P(\mu),-\bigr)$
gives the equivalence of categories
\begin{equation}\label{equiva}
\mathcal{F}^\psi: \O_\psi^\p\cong{A}^\p(\psi)\mod.
\end{equation}

\begin{dfn} \label{basicObj} Let $\psi$ be a dominant integral weight.
For each $\mu\in {{\mathring{W}}^\psi}\cdot\psi$, we define $$
\LFlat[\mu]:=\mathcal{F}^\psi(L(\mu)),
\quad \DFlat[\mu]:=\mathcal{F}^\psi(\Delta(\mu)),
\quad \PFlat[\mu]:=\mathcal{F}^\psi(P(\mu)).
$$
\end{dfn}

By \cite[Theorem 1.1.3]{BGS:Koszul} and \cite[Theorem 1.1]{Backelin:Koszul} (see also \cite[Section 8.1, 8.2, and 8.4]{Strop1}),
$A^\p(\psi)$ can be endowed with a Koszul $\Z$-grading. Let $\widetilde{A}^\p{(\psi)}$ denote the algebra $A^\p(\psi)$ equipped with the Koszul $\Z$-grading. Hence $\widetilde{A}^\p{(\psi)}$ is a positively graded $\mathbb{C}$-algebra with $\widetilde{A}^\p{(\psi)}_0$ spanned by orthogonal idempotents. In particular, simple and projective ${A}^\p{(\psi)}$-modules have graded lifts and every simple $\widetilde{A}^\p{(\psi)}$-module is one-dimensional concentrated in a fixed degree. For each $\mu\in \WLp\cdot\psi$, let $\wLFlat[\mu]$ be the graded lift of simple module $\LFlat[\mu]$ concentrated in degree $0$ and let
\begin{equation}\label{gproj}
\wPFlat[\mu]=\sum_{j=0}^{d_\mu}(\wPFlat[\mu])_j\quad\text{ with $(\wPFlat[\mu])_{d_\mu}\not= 0$}
\end{equation}
be the graded lift of  the projective module $\PFlat[\mu]$ such that the natural projection
from $\wPFlat[\mu]$ to $\wLFlat[\mu]$ is a homogeneous homomorphism of degree $0$.

Let $\lam$ be a dominant integral weight. By \cite[Theorem 1.1.3, Proposition 2.4.1]{BGS:Koszul} and \cite[Theorem 1.1]{Backelin:Koszul}, $\widetilde{A}^\p{(\lam)}$ is Koszul and each indecomposable projective-injective module $P$ is rigid in the sense that both the radical filtration and the socle filtration of $P$ coincide with its grading filtration up to a shift of grading. In particular, the graded length of any indecomposable projective-injective module $P$ is the same as its Loewy length. By Corollary~\ref{symm-eq}, the following proposition obtained by Coulembier and Mazorchuk provides a necessary condition for the existence of a homogeneous symmetrizing form on the endomorphism algebra $B_0^\p$ (see Definition~\ref{2definitions2} below).

 \begin{prop}\cite{CoulMazor}(cf. \cite[Theorem 5.2, Remark 5.3]{MazorStrop}) \label{maincor3.1}
 Every indecomposable projective-injective module of $\widetilde{A}^\p{(\lam)}$ has the same graded length and hence the same Loewy length.
\end{prop}

\begin{rem} Note that a (weaker) block version of the above proposition is a consequence of Proposition \ref{keylem4}. We sketch a proof as follows. It is well known that $\widetilde{A}^\p{(\lam)}\gmod$ is a $\Z$-graded highest weight category with a duality functor.
The double centralizer property in the assumptions of Proposition~\ref{keylem4} holds for the graded algebra $\widetilde{A}^\p{(\lam)}$ by \cite[Theorem 10.1]{Strop2} and \cite[Examples 2.7 (2)]{MazorStrop}. Now all assumptions of Proposition \ref{keylem4} are satisfied and hence the block version of the proposition above follows.
\end{rem}

Using the categorical equivalences given in (\ref{equiva}), we shall simply regard the functors $T_0^\lam$, $T^0_\lam$ as the functors between
${A}^\p(0)\mod$ and $A^\p(\lam)\mod$ without further explanation. Applying \cite[the fourth paragraph in Page 147]{Backelin:Koszul}, both the functors $T_\lam^0$ and $T^\lam_0$ have graded lifts, i.e., we have graded translation
functors: $$\widetilde{T}_\lam^0: \widetilde{A}^\p{(\lam)}\gmod\rightarrow \widetilde{A}^\p{(0)}\gmod\quad \text{and}\quad \widetilde{T}_0^\lam: \widetilde{A}^\p{(0)}\gmod\rightarrow \widetilde{A}^\p{(\lam)}\gmod.$$
Moreover,  $\widetilde{T}_0^\lam$ is a right adjoint functor of $\widetilde{T}_\lam^0$ and
 \begin{equation}\label{gr-tran-si}
    \widetilde{T}_0^\lam\bigl(\wLFlat[w\cdot 0]\bigr)\simeq\wLFlat[w\cdot\lam]\quad\text{for all $w\in\mathring{W}^\lam$.}
 \end{equation}
 Recall that $M\simeq N$ defined in \secref{Graded} denotes that there is a homogeneous isomorphism of degree $0$ between graded modules $M$ and $N$.
  Therefore we have, for all $w,y\in\mathring{W}^\lam$ and $k\in \Z$,
  $$
  \hom_{\widetilde{A}^\p{(0)}}\bigl(\widetilde{T}_\lam^0(\wPFlat[w\cdot\lam]),\wLFlat[y\cdot 0]\<k\>\bigr) \cong\hom_{\widetilde{A}^\p{(\lam)}}\bigl(\wPFlat[w\cdot\lam],\widetilde{T}_0^\lam(\wLFlat[y\cdot 0])\<k\>\bigr) \cong\hom_{\widetilde{A}^\p{(\lam)}}\bigl(\wPFlat[w\cdot\lam],\wLFlat[y\cdot\lam]\<k\>\bigr) \cong\delta_{k,0}\delta_{w,y}\C
$$
and hence
\begin{equation}\label{gr-tran-proj}
 \widetilde{T}_\lam^0(\wPFlat[w\cdot\lam])\simeq\wPFlat[w\cdot 0]\quad\text{for all $w\in\mathring{W}^\lam$.}
\end{equation}

The following lemma is a graded version of Lemma~\ref{tlam0simple}. Note that $d_{\mu}$ is defined in (\ref{gproj}).

\begin{lem} \label{keylem41} For $w\in \check{W}^\lam$, we have
\begin{itemize}
  \item[(i)] $\head(\widetilde{T}_\lam^0\bigl(\wLFlat[w\cdot\lam]\bigr))\simeq\wLFlat[w\cdot 0]$, $\soc(\widetilde{T}_\lam^0\bigl(\wLFlat[w\cdot\lam]\bigr))\simeq\wLFlat[w\cdot 0]\<d_{w\cdot 0}-d_{w\cdot \lam}\>$;
  \item[(ii)] $\widetilde{T}_\lam^0\bigl(\wLFlat[w\cdot\lam]\bigr)$ is rigid, and both the radical filtration and the socle filtration of $\widetilde{T}_\lam^0\bigl(\wLFlat[w\cdot\lam]\bigr)$ coincide with its grading filtration (up to a grading shift), and the graded length of $\widetilde{T}_\lam^0\bigl(\wLFlat[w\cdot\lam]\bigr)$ is $d_{w\cdot 0}-d_{w\cdot \lam}+1$. In particular, $d_{w\cdot 0}\geq d_{w\cdot \lam}$;

  \item[(iii)] $\bigl[\widetilde{T}_\lam^0\bigl(\wLFlat[w\cdot\lam]\bigr):\wLFlat[w\cdot 0]\<k\>\bigr]\neq 0$ only if $0\leq k\leq d_{w\cdot 0}-d_{w\cdot \lam}$. Moreover, $$
\bigl[\widetilde{T}_\lam^0\bigl(\wLFlat[w\cdot\lam]\bigr):\wLFlat[w\cdot 0]\bigr]=1=\bigl[\widetilde{T}_\lam^0\bigl(\wLFlat[w\cdot\lam]\bigr):\wLFlat[w\cdot 0]\<d_{w\cdot 0}-d_{w\cdot \lam}\>\bigr].
$$
\end{itemize}
\end{lem}

\begin{proof} Applying Lemma \ref{tlam0simple}, we see that $\head({T}_\lam^0\bigl(\LFlat[w\cdot\lam]\bigr))\cong\LFlat[w\cdot 0]$ and $\soc({T}_\lam^0\bigl(\LFlat[w\cdot\lam]\bigr))\cong\LFlat[w\cdot 0]$. Now the degree $0$ surjection $\wPFlat[w\cdot\lam]\twoheadrightarrow
\wLFlat[w\cdot\lam]$ naturally induces a degree $0$ surjection $\wPFlat[w\cdot 0]\simeq\widetilde{T}_\lam^0\bigl(\wPFlat[w\cdot\lam]\bigr)\twoheadrightarrow
\widetilde{T}_\lam^0\bigl(\wLFlat[w\cdot\lam]\bigr)$, which implies that $\head(\widetilde{T}_\lam^0\bigl(\wLFlat[w\cdot\lam]\bigr))\simeq\wLFlat[w\cdot 0]$.
Similarly, the degree $0$ embedding $\wLFlat[w\cdot\lam]\<d_{w\cdot \lam}\>\hookrightarrow\wPFlat[w\cdot\lam]$ naturally induces a degree $0$ embedding
$$
\widetilde{T}_\lam^0\bigl(\wLFlat[w\cdot\lam]\bigr)\<d_{w\cdot \lam}\>=\widetilde{T}_\lam^0\bigl(\wLFlat[w\cdot\lam]\<d_{w\cdot \lam}\>\bigr)
\hookrightarrow\widetilde{T}_\lam^0\bigl(\wPFlat[w\cdot\lam]\bigr)\simeq
\wPFlat[w\cdot 0] .
$$
Since $\soc(\wPFlat[w\cdot 0])=\wLFlat[w\cdot 0]\<d_{w\cdot 0}\>$, it follows that $\soc(\widetilde{T}_\lam^0\bigl(\wLFlat[w\cdot\lam]\bigr))\simeq\wLFlat[w\cdot 0]\<d_{w\cdot 0}-d_{w\cdot \lam}\>$ as required.
This proves (i).

From above, $\widetilde{T}_\lam^0\bigl(\wLFlat[w\cdot\lam]\bigr)$ has a unique simple head as well as a unique simple socle, we see that (ii) follows from \cite[Proposition 2.4.1]{BGS:Koszul}, the fact that $\widetilde{A}^\p{(0)}$ is Koszul (\cite[Theorem 1.1.3]{BGS:Koszul}) and (i). Finally, (iii) follows from (ii) and (i).
\end{proof}

The following Lemma is a consequence of Lemma~\ref{keylem41}, Lemma~\ref{5properties}(i), (ii) and \eqref{gr-tran-si}.

\begin{lem} \label{finalcor} For $w\in \check{W}^\lam$,  there exist a sequence of integers $0=a_1^w< a_2^w\leq \cdots\leq a_{|W_\lam|-1}^w<a_{|W_\lam|}^w=d_{w\cdot 0}-d_{w\cdot \lam}$ such that
\begin{equation*}
\widetilde{T}_0^\lam \widetilde{T}^0_\lam(\wLFlat[w\cdot\lam])\simeq
\bigoplus_{j=1}^{|W_\lam|}\wLFlat[w\cdot\lam]\<a_j^w\>
\quad \text{and}\quad
\widetilde{T}_0^\lam(\wPFlat[w\cdot 0])\simeq \widetilde{T}_0^\lam \widetilde{T}^0_\lam(\wPFlat[w\cdot\lam])\simeq
\bigoplus_{j=1}^{|W_\lam|}\wPFlat[w\cdot\lam]\<a_j^w\>.
\end{equation*}
\end{lem}

\begin{dfn} \label{2definitions2}
For a dominant integral weight $\psi$, the endomorphism algebra of the basic projective-injective module of $\O_\psi^\p$ is denoted by $$
 {B}^\p_\psi:=\End_{\widetilde{A}^\p{(\psi)}}\bigl(\oplus_{w\in\check{W}^\psi}\wPFlat[w\cdot \psi]\bigr).
$$
\end{dfn}
It is clear that $B^\p_\psi$ is a positively graded $\C$-algebra.

Since $B^\p_0$ is a symmetric algebra over $\C$ \cite[Theorem 4.6]{MazorStrop}, there are an admissible basis $\mathcal{B}$ and the  canonical symmetrizing form $\tr$ attached to $\mathcal{B}$ on $B^\p_0$ by Corollary~\ref{symm-eq} and Proposition~\ref{keylem3}. For each
 $w\in\check{W}^0$, there is a unique endomorphism $\theta_{w\cdot 0}\in\mathcal{B}\cap\End_{\widetilde{A}^\p{(0)}}(\wPFlat[w\cdot 0])_{d_{w\cdot 0}}$. Note that  $\theta_{w\cdot 0}(\wPFlat[w\cdot 0])=\soc(\wPFlat[w\cdot 0])\simeq\wLFlat[w\cdot 0]\<{d_{w\cdot 0}}\>$. We may assume that  $\tr(\theta_{w\cdot 0})=1$ for all $w\in\check{W}^0$.

Recall that $\check{W}^\lam\subset\check{W}^0$ by \eqref{2inclusion}. From now on, we will fix an isomorphism $\widetilde{T}_0^\lam (\wPFlat[w\cdot 0])\simeq
\bigoplus_{j=1}^{|W_\lam|}\Pw_j$ for each $w\in\check{W}^\lam$ obtained  in Lemma~\ref{finalcor}, where $\Pw_j:=\wPFlat[w\cdot\lam]\<a_j^w\>$ for all $1\le j\le |W_\lam|$.  For $w\in\check{W}^\lam$, we will use the matrix notations to write any element in
$\Hom_{\widetilde{A}^\p{(\lam)}}(\bigoplus_{j=1}^{|W_\lam|}\Pw_j, \bigoplus_{j=1}^{|W_\lam|}\Pw_j)$ in the form of
$$\sum_{1\le i,j\le|W_\lam|}h_{ji},\quad\text{where $h_{ji}\in
\Hom_{\widetilde{A}^\p{(\lam)}}(\Pw_j, \Pw_i)$. }$$
Therefore every element in
$\End_{\widetilde{A}^\p{(\lam)}}(\bigoplus_{j=1}^{|W_\lam|}\Pw_j)$ can be written in the form
\begin{equation}\label{exp}
\sum_{1\le i,j\le|W_\lam|}g_{ji}[a_i^w]\pi^w_{ji},
\end{equation}
where $g_{ji}\in \End_{\widetilde{A}^\p{(\lam)}}(\wPFlat[w\cdot\lam])$ and $\pi^w_{ji}:=\id_{\wPFlat[w\cdot\lam]}[a_j^w,a_i^w]$, where $\id_{\wPFlat[w\cdot\lam]}$ is the identity map on ${\wPFlat[w\cdot\lam]}$, and $f[j,i]$ and $f[i]$ are defined in \eqref{gr-shift} and \eqref{gr-shift=} for a homogeneous homomorphism  $f$ between graded modules. In other words, $\pi^w_{ji}$  is a homogeneous homomorphism from $\Pw_j$ onto $\Pw_i$ of degree $a_i^w-a_j^w$ such that $\underline{\mathrm{For}}(\pi^w_{ji})$ is the identity map on $\PFlat[w\cdot\lam]$ for all $i$, $j$.

Since $\deg(\theta_{w\cdot 0})=d_{w\cdot 0}$ for $w\in\check{W}^\lam$, we have $\deg(\widetilde{T}_0^\lam(\theta_{w\cdot 0}))=d_{w\cdot 0}$ and hence
\begin{equation}\label{T-theta}
\widetilde{T}_0^\lam(\theta_{w\cdot 0})=\theta_{w\cdot \lam}[d_{w\cdot 0}-d_{w\cdot \lam}]\pi^w_{1,|W_\lam|}
\end{equation}
by Lemma~\ref{finalcor} and the fact that $0\leq\deg f\leq d_{w\cdot\lam}$ for any homogeneous map $f\in\End_{\widetilde{A}^\p{(\lam)}}(\PFlat[w\cdot\lam])$, where $\theta_{w\cdot \lam}\in  \End_{\widetilde{A}^\p{(\lam)}}(\wPFlat[w\cdot \lam])_{d_{w\cdot \lam}}$ satisfies $\theta_{w\cdot \lam}(\wPFlat[w\cdot \lam])=\soc(\wPFlat[w\cdot \lam])$.
The $\theta_{w\cdot \lam}$s defined from above play crucial roles for the rest of the paper.

By \cite[Theorem 4.6]{MazorStrop}, $B^\p_0$ is a symmetric algebra over $\C$. By Proposition~\ref{homogekeylem2}, for every $w\in\check{W}^\lam$ there is a homogeneous homomorphism $\fw\in\End_{\widetilde{A}^\p{(0)}}(\wPFlat[w\cdot 0])_{d_{w\cdot 0}-d_{w\cdot \lam}}$ such that
 \begin{equation}\label{thetabar}
\fw\widetilde{T}_\lam^0(\theta_{w\cdot\lam})=\theta_{w\cdot 0}=\widetilde{T}_\lam^0(\theta_{w\cdot\lam})\fw.
\end{equation}
Let $$\fwn:=\sum_{w\in\check{W}^\lam}\fw.$$

\begin{lem}\label{t} For each $w\in\check{W}^\lam$, we can write
$$\widetilde{T}_0^\lam(\fw)=
\sum_{1\le i,j\le|W_\lam|}q^w_{ji}[a_i^w]\pi^w_{ji},
$$
where $q^w_{ji}\in \End_{\widetilde{A}^\p{(\lam)}}(\wPFlat[w\cdot \lam])$ such that $\deg(q^w_{ji})>0$ or $q^w_{ji}=0$ for all $(j,i)\not=(1,|W_\lam|)$ and $q^w_{1,|W_\lam|}$ is the identity map.
\end{lem}

\begin{proof} For each $w\in\check{W}^\lam$, we can write
$$\widetilde{T}_0^\lam(\fw)=
\sum_{1\le i,j\le|W_\lam|}q^w_{ji}[a_i^w]\pi^w_{ji},
$$
where $q^w_{ji}\in \End_{\widetilde{A}^\p{(\lam)}}(\wPFlat[w\cdot \lam])$ for all $i,j$ by (\ref{exp}).
Note that $\widetilde{T}_0^\lam\widetilde{T}_\lam^0(\theta_{w\cdot \lam})=\sum_{1\le j\le|W_\lam|}\theta_{w\cdot \lam}[a_j^w]$ by Lemma~\ref{5properties} and Lemma~\ref{finalcor}. By (\ref{T-theta}) and (\ref{thetabar}), we have
 \begin{align*}
\theta_{w\cdot \lam}[d_{w\cdot 0}-d_{w\cdot \lam}]\pi^w_{1,|W_\lam|}
&=\widetilde{T}_0^\lam(\theta_{w\cdot 0})\\
&=\widetilde{T}_0^\lam\widetilde{T}_\lam^0(\theta_{w\cdot\lam})\widetilde{T}_0^\lam(\fw)\\
&=\sum_{1\le j\le|W_\lam|}\theta_{w\cdot \lam}[a_j^w]\sum_{1\le i,j\le|W_\lam|}q^w_{ji}[a_i^w]\pi^w_{ji}\\
&=\sum_{1\le i,j\le|W_\lam|}\theta_{w\cdot \lam}[a_i^w]q^w_{ji}[a_i^w]\pi^w_{ji}\\
&=\sum_{1\le i,j\le|W_\lam|}(\theta_{w\cdot \lam}q^w_{ji})[a_i^w]\pi^w_{ji}.
\end{align*}
Therefore $(\theta_{w\cdot \lam}q^w_{ji})[a_i^w]=0$ for all $(j,i)\not=(1,|W_\lam|)$  and $(\theta_{w\cdot \lam}q^w_{ji})[a_i^w]=\theta_{w\cdot \lam}[d_{w\cdot 0}-d_{w\cdot \lam}]$ for $(j,i)=(1,|W_\lam|)$. Hence $\deg(q^w_{ji})>0$ or $q^w_{ji}=0$ for all $(j,i)\not=(1,|W_\lam|)$ and $q^w_{1,|W_\lam|}$ is the identity map.
\end{proof}

The following proposition will play a key role in the proof of Theorem \ref{mainthm1}.

\begin{prop} \label{keyprop} For $w,y\in\check{W}^\lam$, we have
\begin{equation}\label{goal-eq}
 \widetilde{T}_\lam^0(f)\widetilde{T}_\lam^0(h)\fwn_{y}
=\widetilde{T}_\lam^0(f)\fwn_{w}\widetilde{T}_\lam^0(h),
\end{equation}
for  all $f\in\Hom_{\widetilde{A}^\p{(\lam)}}\bigl(\wPFlat[w\cdot\lam],\wPFlat[y\cdot\lam]\bigr)_k$, $h\in\Hom_{\widetilde{A}^\p{(\lam)}}\bigl(\wPFlat[y\cdot\lam],\wPFlat[w\cdot\lam]\bigr)_{d_{y\cdot\lam}-k}$ and $k\in\Z$.

Moreover, if $\Hom_{\widetilde{A}^\p{(\lam)}}\bigl(\wPFlat[w\cdot\lam],\wPFlat[y\cdot\lam]\bigr)\not=0$, then $d_{w\cdot \lam}=d_{y\cdot \lam}$.
\end{prop}

\begin{proof}
We first show $\widetilde{T}^\lam_0\Bigl(\widetilde{T}_\lam^0(f)\widetilde{T}_\lam^0(h)\fwn_{y}\Bigr)$ and $\widetilde{T}^\lam_0\Bigl(\widetilde{T}_\lam^0(f)\fwn_{w}\widetilde{T}_\lam^0(h)\Bigr)$ are equal. By Lemma~\ref{t}, we have $\widetilde{T}_0^\lam(\fwn_{y})=\sum_{1\le i,j\le|W_\lam|}q^y_{ji}[a_i^y]\pi^y_{ji}$ and
\begin{align*}
\widetilde{T}^\lam_0\Bigl(\widetilde{T}_\lam^0(f)\widetilde{T}_\lam^0(h)\fwn_{y}\Bigr)
 & =\widetilde{T}^\lam_0\widetilde{T}_\lam^0(fh)\widetilde{T}^\lam_0(\fwn_{y})\\
 &=\big(\sum_{1\le i\le|W_\lam|}(fh)[a_i^y]\big)\big(\sum_{1\le i,j\le|W_\lam|}q^y_{ji}[a_i^y]\pi^y_{ji}\big)\\
  &=\sum_{1\le i,j\le|W_\lam|}(fhq^y_{ji})[a_i^y]\pi^y_{ji}\\
    &=(fhq^y_{1,|W_\lam|})[a^y_{|W_\lam|}]\pi^y_{1,|W_\lam|}\\
    &=(fh)[a^y_{|W_\lam|}]\pi^y_{1,|W_\lam|}.
\end{align*}
The second equality follows from Lemma~\ref{5properties}(i) and Lemma~\ref{finalcor}, the fourth equality follows from $fhq^y_{ji}=0$ for all $(j,i)\not=({1,|W_\lam|})$ because $\deg(fhq^y_{ji})>d_{y\cdot\lam}$ for all $(j,i)\not=({1,|W_\lam|})$ by Lemma~\ref{t} and the assumptions, the final equality follows from $q^y_{1,|W_\lam|}$ is the identity map  by Lemma~\ref{t}.
Similarly,
\begin{align*}
\widetilde{T}^\lam_0\Bigl(\widetilde{T}_\lam^0(f)\fwn_{w}\widetilde{T}_\lam^0(h)\Bigr)
  &=\big(\sum_{1\le i\le|W_\lam|}f[a^w_i,a^y_i]\big) \big(\sum_{1\le i,j\le|W_\lam|}q^w_{ji}[a_i^w]\pi^w_{ji}\big) \big(\sum_{1\le j\le|W_\lam|}h[a^y_j, a^w_j]\big)\\
  &=\sum_{1\le i,j\le|W_\lam|}f[a^w_i,a^y_i]q^w_{ji}[a_i^w]\pi^w_{ji}h[a^y_j, a^w_j]\\
    &=f[a^w_{|W_\lam|},a^y_{|W_\lam|}]q^w_{1,|W_\lam|}[a_{|W_\lam|}^w]\pi^w_{1,|W_\lam|}h[0,0]\\
       &=f[a^w_{|W_\lam|},a^y_{|W_\lam|}]\pi^w_{1,|W_\lam|}h[0,0]\\
    &=(fh)[a^y_{|W_\lam|}]\pi^y_{1,|W_\lam|}.
\end{align*}
The first equality follows from Lemma~\ref{5properties}(i), Lemma~\ref{finalcor}, Lemma~\ref{t} and the assumptions, the third equality follows from $fq^w_{ji}h=0$ for all $(j,i)\not=({1,|W_\lam|})$ because  $\deg(fq^w_{ji}h)>d_{y\cdot\lam}$ for all $(j,i)\not=({1,|W_\lam|})$ by Lemma~\ref{t}, the fourth equality follows from $q^w_{1,|W_\lam|}$ is the identity map  by Lemma~\ref{t} and the final equality follows from $\underline{\mathrm{For}}(\pi^y_{1,|W_\lam|})=\id_{\PFlat[y\cdot \lam]}$ and $\underline{\mathrm{For}}(\pi^w_{1,|W_\lam|})=\id_{\PFlat[w\cdot \lam]}$. Therefore $\widetilde{T}^\lam_0\Bigl(\widetilde{T}_\lam^0(f)\widetilde{T}_\lam^0(h)\fwn_{y}\Bigr)$ and $\widetilde{T}^\lam_0\Bigl(\widetilde{T}_\lam^0(f)\fwn_{w}\widetilde{T}_\lam^0(h)\Bigr)$ are equal.

Now we claim that $\widetilde{T}_\lam^0(f)\fwn_{w}\widetilde{T}_\lam^0(h)\in \C\theta_{y\cdot 0}$. Otherwise, the head and the socle of the image of $\underline{\mathrm{For}}\Bigl(\widetilde{T}_\lam^0(f)\fwn_{w}\widetilde{T}_\lam^0(h)\Bigr)$ are isomorphic to $\LFlat[w\cdot 0]$ such that they are not equal. We get a contradiction, by using Lemma~\ref{5properties}(ii), to the result from computation above that the image of $\widetilde{T}^\lam_0\Bigl(\widetilde{T}_\lam^0(f)\fwn_{w}\widetilde{T}_\lam^0(h)\Bigr)$ is a simple module.  Hence $\widetilde{T}_\lam^0(f)\fwn_{w}\widetilde{T}_\lam^0(h)\in \C\theta_{y\cdot 0}$. We also have $\widetilde{T}_\lam^0(f)\widetilde{T}_\lam^0(h)\fwn_{y}\in \C\theta_{y\cdot 0}$ because $\widetilde{T}_\lam^0(f)\widetilde{T}_\lam^0(h)\fwn_{y}=\widetilde{T}_\lam^0(c\theta_{y\cdot\lam})\fwn_{y} =c\theta_{y\cdot 0}$ by the assumption and \eqref{thetabar} for some $c\in \C$.
We have $\widetilde{T}_\lam^0(f)\widetilde{T}_\lam^0(h)\fwn_{y}=\widetilde{T}_\lam^0(f)
\fwn_{w}\widetilde{T}_\lam^0(h)$ because they belong to $\C\theta_{y\cdot 0}$ and $\widetilde{T}^\lam_0\Bigl(\widetilde{T}_\lam^0(f)\widetilde{T}_\lam^0(h)\fwn_{y}\Bigr)=\widetilde{T}^\lam_0\Bigl(\widetilde{T}_\lam^0(f)
\fwn_{w}\widetilde{T}_\lam^0(h)\Bigr)$.
This completes the proof of the first part of the proposition.

For the second part, we may assume $f\in \Hom_{\widetilde{A}^\p{(\lam)}}\bigl(\wPFlat[w\cdot\lam],\wPFlat[y\cdot\lam]\bigr)_k$ for some $k\ge 0$ such that $f\not=0$. There is a homomorphism $h\in\Hom_{\widetilde{A}^\p{(\lam)}}\bigl(\wPFlat[y\cdot\lam],\wPFlat[w\cdot\lam]\bigr)_{d_{y\cdot\lam}-j}$ such that $fh=\theta_{y\cdot\lam}$. Hence $\widetilde{T}_\lam^0(f)\widetilde{T}_\lam^0(h)\fwn_{y}=\theta_{y\cdot 0}$.

Since $B^\p_0$ is a symmetric algebra over $\C$ \cite[Theorem 4.6]{MazorStrop} and $\widetilde{A}^\p{(0)}$ is Koszul (\cite[Theorem 1.1.3]{BGS:Koszul}), by Corollary \ref{symm-eq} we know that $B^\p_0=\End_{\widetilde{A}^\p{(0)}}\bigl(\oplus_{w\in\check{W}^0}\wPFlat[w\cdot 0]\bigr)$ has an admissible basis and the canonical form $\tr$ attached to the basis is a symmetrizing form. Hence
\begin{align*}
\tr\Bigl(\widetilde{T}_\lam^0(hf)\fwn_{w}\Bigr)
&=\tr\Bigl(\fwn_{w}\widetilde{T}_\lam^0(h)\widetilde{T}_\lam^0(f)\Bigr)\\
&= \tr \Bigl(\widetilde{T}_\lam^0(f)\fwn_{w}\widetilde{T}_\lam^0(h)\Bigr) \\
   &= \tr \Bigl(\widetilde{T}_\lam^0(f)\widetilde{T}_\lam^0(h)\fwn_{y}\Bigr)\\
  &=\tr\Bigl(\theta_{y\cdot 0}\Bigr)\not=0.
\end{align*}
Since $\widetilde{T}_\lam^0(hf)\fwn_{w}\in\End_{\widetilde{A}^\p{(0)}}(\wPFlat[w\cdot 0])$ is homogeneous, it follows that $\widetilde{T}_\lam^0(hf)\fwn_{w}\in \C^\times\theta_{w\cdot 0}$ and $\deg(hf)=d_{w\cdot 0}- (d_{w\cdot 0}-d_{w\cdot \lam})=d_{w\cdot \lam}$. Also $\deg(h)+\deg(f)=d_{y\cdot \lam}$. This implies $d_{y\cdot \lam}=d_{w\cdot \lam}$.
This completes the proof.
\end{proof}

\begin{rem} We conjecture that the element $\fwn=\sum_{w\in\check{W}^\lam}\overline{\theta}_w$ commutes with the element ${T}_\lam^0(h)$ for any $h\in\End_{\widetilde{A}^\p{(\lam)}}\bigl(\oplus_{w\in\check{W}^\lam}\wPFlat[w\cdot\lam]\bigr)$.
\end{rem}

\medskip
\noindent
{\textbf{Proof of Theorem \ref{mainthm1}:}} By Corollary~\ref{kmu}, it is enough to show that $B^\p_\lam =\End_{\widetilde{A}^\p{(\lam)}}\bigl(\oplus_{w\in\check{W}^\lam}\wPFlat[w\cdot\lam]\bigr)$ is a symmetric algebra.
Recall that $
B^\p_0=\End_{\widetilde{A}^\p{(0)}}\bigl(\oplus_{w\in\check{W}^0}\wPFlat[w\cdot 0]\bigr)
$.

Since $B^\p_0$ is a symmetric algebra over $\C$ \cite[Theorem 4.6]{MazorStrop} and $\widetilde{A}^\p{(0)}$ is Koszul (\cite[Theorem 1.1.3]{BGS:Koszul}), by Corollary \ref{symm-eq} we know that $B^\p_0$ has an admissible basis and the canonical form $\tr$ attached to the basis is symmetrizing form.

Recall that $\check{W}^\lam\subset\check{W}^0$ by (\ref{2inclusion}). Now we define a form $\tr_\lam$ on $B^\p_\lam$ by $\tr_\lam(g)=\tr(\widetilde{T}_\lam^0(g)\fwn)$ for all $g\in B^\p_\lam$. To show $\tr_\lam$ is non-degenerate, it is enough to show that given $w,y\in \check{W}^\lam$ and a nonzero homomorphism $g\in\Hom_{\widetilde{A}^\p{(\lam)}}\bigl(\wPFlat[w\cdot\lam],\wPFlat[y\cdot\lam]\bigr)_j$ for some $j$, there is  an $h\in B^\p_\lam$ such that $\tr_\lam(gh)\not=0$.

By Lemma~\ref{G-keylemma0}, there exists $h\in \Hom_{\widetilde{A}^\p{(\lam)}}\bigl(\wPFlat[y\cdot\lam],\wPFlat[w\cdot\lam]\bigr)_{d_{y\cdot\lam}-j}$ such that $gh=\theta_{y\cdot \lam}$. Therefore $\tr_\lam(gh)=\tr_\lam(\theta_{y\cdot \lam})=\tr\bigl(\widetilde{T}_{\lam}^{0}(\theta_{y\cdot \lam})\fwn\bigr)=\tr(\theta_{y\cdot 0})=1$. Hence $\tr_\lam$ is non-degenerate. It remains to show that $\tr_\lam$ is symmetric.

Let $g$ be  a nonzero homomorphism in $\Hom_{\widetilde{A}^\p{(\lam)}}\bigl(\wPFlat[w\cdot\lam],\wPFlat[y\cdot\lam]\bigr)_j$ for $w,y\in \check{W}^\lam$ and $j\in\Z$. By Proposition~\ref{maincor3.1} or Proposition~\ref{keyprop}, we have $d_{w\cdot\lam}=d_{y\cdot\lam}$. It is clear that $\tr_\lam(gh)=\tr_\lam(hg)=0$ for all $h\in\Hom_{\widetilde{A}^\p{(\lam)}}\bigl(\wPFlat[y'\cdot\lam],\wPFlat[w'\cdot\lam]\bigr)_k$ for $w',y'\in \check{W}^\lam$ with $(w', y', k)\not=(w, y, d_{y\cdot\lam}-j)$.  For $h\in\Hom_{\widetilde{A}^\p{(\lam)}}\bigl(\wPFlat[y\cdot\lam],\wPFlat[w\cdot\lam]\bigr)_{d_{y\cdot\lam}-j}$,
\begin{align*}
 \tr_\lam(gh) & =\tr(\widetilde{T}_\lam^0(gh)\fwn) =\tr(\widetilde{T}_\lam^0(gh)\fwn_{y})=\tr(\widetilde{T}_\lam^0(g)\widetilde{T}_\lam^0(h)\fwn_{y})\\
    & =\tr(\widetilde{T}_\lam^0(g)\fwn_{w}\widetilde{T}_\lam^0(h))=\tr(\widetilde{T}_\lam^0(h)\widetilde{T}_\lam^0(g)\fwn_{w}) =\tr(\widetilde{T}_\lam^0(hg)\fwn_{w})\\
       &=\tr_\lam(hg),
       \end{align*}
where we have used Proposition~\ref{keyprop} for the fourth equality and the fifth equality follows from the fact that $\tr$ is a symmetrizing form. Therefore $\tr_\lam$ is also a symmetrizing form. This completes the proof of Theorem \ref{mainthm1}.
\medskip

\begin{cor} \label{maincor3} For a dominant integral weight $\lam$  and $P$ a projective-injective module in $\O_\lam^\p$, the endomorphism algebra $\End_{\O_\lam^{\p}}\bigl(P\bigr)$ is a graded symmetric algebra equipped with a homogeneous non-degenerate symmetric bilinear form of degree $1-d$, where $d$ is the common graded length of all indecomposable projective-injective modules in $\O_\lam^\p$.
\end{cor}

\begin{proof} The corollary follows from Theorem~\ref{mainthm1} and Corollary~\ref{kmu}.
\end{proof}

Finally, by the main results in \cite{AST}, we know that the endomorphism algebra of any tilting module in $\O_\lam^\p$ has a cellular structure. This can be generalized to the $\Z$-graded setting without difficulty. Note that every tilting module in $\O_\lam^\p$ has a graded lift \cite[Corollary 5]{MazorOvsi}.

\begin{prop} \label{mainlem3} \text{(cf. \cite{AST})} Let $\lam$ be a dominant integral  weight and $Q$ a tilting module in $\O_\lam^\p$. Then the endomorphism algebra $\End_{\O_\lam^{\p}}\bigl(Q\bigr)$ is a $\Z$-graded cellular algebra over $\C$ in the sense of \cite{HuMathas:GradedCellular}.
\end{prop}

In fact, when $\mathfrak{g}$ is a semisimple Lie algebra  of type $A$, Brundan and Kleshchev \cite{BK:HigherSchurWeyl} show that the endomorphism algebra of the basic projective-injective module in the parabolic BGG category $\O^\p$ is essentially the basic algebra of the cyclotomic quiver Hecke algebra associated to the linear quiver (i.e., the degenerate cyclotomic Hecke algebras of type $A$). So Corollary \ref{maincor3} and Proposition~\ref{mainlem3} indicate that for any semisimple Lie algebra $\mathfrak{g}$, the endomorphism algebra of the basic projective-injective module in the parabolic BGG category $\O^\p$ behaves very much like the cyclotomic quiver Hecke algebra (cf. \cite{BK:HigherSchurWeyl}, \cite{HuMathas:GradedCellular}, \cite[Proposition 3.10]{SVV}) and is a new class of interesting objects which deserves further study.

\bigskip

\section*{Acknowledgements}

The first author was supported by the National Natural Science Foundation of China (No. 11525102, 11471315).
The second author was partially supported by MoST grant 104-2115-M-006-015-MY3 of Taiwan.
\bigskip

\bigskip

\appendix
\def\theequation{\Alph{section}\arabic{equation}}
\def\theProposition{\Alph{section}\arabic{equation}}
\def\theLemma{\Alph{section}\arabic{equation}}
\def\theTheorem{\Alph{section}\arabic{equation}}
\def\theCorollary{\Alph{section}\arabic{equation}}

\section{}

In this appendix a short proof of Theorem \ref{mainthm1} is given. The proof is due to Kevin Coulembier and
Volodymyr Mazorchuk who kindly forwarded the details to us and allowed us to include it in this article after a first version of the paper was submitted.

\begin{lem} \label{App1} Let $A, B$ be two finite dimensional associative algebras over a field $K$. Let $F: A\mod\rightarrow B\mod$ and $G: B\mod\rightarrow A\mod$ be a biadjoint pair of functors such that $FG$ is isomorphic to a direct sum of $r$ copies of the identity endofunctor $\Id_{B}$ of $B$-mod, where $r\in\Z_{>0}$. Then, if $A$ is symmetric, then so is $B$.
\end{lem}

\begin{proof} Let us start by observing that \begin{enumerate}
\item the class of symmetric algebras is closed under Morita equivalence (this follows immediately e.g. from Lemma 3.1(2) of \cite{MazorStrop});
\item if an algebra $Q$ is symmetric and $e\in Q$ is an idempotent in $Q$, then $eQe$ is symmetric (this follows by restricting the symmetric trace form from
$Q$ to $eQe$).
\end{enumerate}

Now, the functor $G$, being biadjoint to another functor, is exact and sends projectives to projectives. Therefore $G({}_{B}B)$ is a projective $A$-module and for every $B$-module $M$, $G(M)$ has a presentation with modules being the sum of copies of $G({}_{B}B)$, where ${}_{B}B$ denotes the $K$-vector space $B$ equipped with the natural left $B$-module structure. Using the two observations above together with  \cite[Proposition 5.2 and 5.3]{Aus1}, without loss of generality we may assume that $G({}_{B}B) \cong {}_{A}A$. Since $FG\cong\Id_{B}^{\oplus r}$, the functor $G$ is faithful. By functoriality, $G$ thus defines an algebra monomorphism from $B^{\op}= \End_{B}({}_{B}B)$ to $A^{\op}= \End_{A}({}_{A}A)$. Hence $G$ identifies $B$ with a unital subalgebra of $A$.

Being exact, $G$ is uniquely defined by its value on ${}_{B}B$ and $\End_{B}({}_{B}B)$. From the identification of $B$ as a subalgebra of $A$, we have that $G$ is isomorphic to induction from $B$ to $A$ by Eilenberg-Watts theorem (see, for example, \cite[Chapter II, Theorem 2.3]{Bass}), that is, $G\cong{}_{A}A_{B}\otimes_{B}-$. By adjointness, $F$ is then isomorphic to the restriction from $A$ to $B$, that is, $F\cong{}_{B}A_{A}\otimes_{A}-$. Therefore $FG\cong{}_{B}A_{B}\otimes_{B}-$.  As $FG\cong\Id_{B}^{\oplus r}$, the $B$-$B$-bimodule ${}_{B}A_{B}$, which represents $FG$, is isomorphic to a direct
sum of $r$ copies of ${}_{B}B_{B}$ by \cite[Chapter II, Proposition 2.2]{Bass}. Since $A$ is symmetric, we have ${}_{A}A_{A}\cong {}_{A}A^{\ast}_{A}$ as an $A$-$A$-bimodule and hence $$
\underbrace{{}_{B}B_{B}\oplus\cdots\oplus{}_{B}B_{B}}_{\text{$r$ copies}}\cong\underbrace{{}_{B}B^{\ast}_{B}\oplus\cdots\oplus{}_{B}B^{\ast}_{B}}_{\text{$r$ copies}}.$$
 This implies ${}_{B}B_{B}\cong {}_{B}B^{\ast}_{B}$ as $B$-$B$-bimodules by Krull-Schmidt Theorem. Therefore $B$ is symmetric.
\end{proof}

For a dominant integral weight $\psi$ of a complex semisimple Lie algebra $\g$, we recall that $
A^\p(\psi):=\biggl(\End_{\O^\p}\Bigl(\oplus_{\mu\in {{\mathring{W}}^\psi}}P(\mu)\Bigr)\biggr)^{\mathrm{op}}
$ denotes the basic algebra of $\O_\psi^\p$ defined in Definition~\ref{KoszulBasic} and  $\Lam_0^\psi=\bigl\{w\cdot\psi | w\in\check{W}^\psi\bigr\}$ denotes the set of socular weight of $\O_\psi^\p$ defined in (\ref{2defn}). Let ${B}^\p(\psi):=\biggl(\End_{\O^\p}\bigl(\oplus_{\mu\in\Lam_0^\psi}P(\mu)\bigr)\biggl)^{\mathrm{op}}$ denote the opposite algebra of  the endomorphism algebra of the basic projective-injective module of $\O_\psi^\p$.

For a dominant integral weight $\lam$, let $\mathcal{C}_\lam$ (resp. $\mathcal{C}'_\lam$) be the full subcategories of $A^\p(\lam)\mod$ (resp. $A^\p(0)\mod$) consisting of modules with a presentation $Q_1\rightarrow Q_0\rightarrow M\rightarrow 0$ such that $Q_0$ and $Q_1$ are sum of summands of the basic projective-injective $A^\p(\lam)$-module (resp., the projective-injective $A^\p(0)$-module $\bigoplus_{w\in\check{W}^\lam}\PFlat[w\cdot 0]$).  By Lemma~\ref{5properties} (i) and (iii), $T_0^\lam$ sends $\mathcal{C}'_\lam$ to $\mathcal{C}_\lam$ and $T_\lam^0$ sends $\mathcal{C}_\lam$ to $\mathcal{C}'_\lam$. Note that the restriction of the functor $T_0^\lam$ to $\mathcal{C}'_\lam$ and the restriction of the functor $T_\lam^0$ to $\mathcal{C}_\lam$ are a biadjoint pair of exact functor. By  \cite[Proposition 5.2 and 5.3]{Aus1}, there are equivalences of categories: $\mathcal{C}_\lam\simeq {B}^\p(\lam)\mod$ and $\mathcal{C}'_\lam\simeq \hat{B}^\p(0)\mod$,
where $\hat{B}^\p(0):=\biggl(\End_{\O^\p}\bigl(\oplus_{w\in\check{W}^\lam}\PFlat[w\cdot 0]\bigr)\biggl)^{\mathrm{op}}$. Recall the observation of (1) and (2) in the proof of Lemma~\ref{App1}. Now Theorem~\ref{mainthm1} follows from \cite[Theorem 4.6]{MazorStrop} and  Lemma~\ref{App1}.

\end{document}